\documentclass{amsart}
\usepackage{amssymb,amstext,amsmath,amscd,amsthm,amsfonts,enumerate,graphicx,latexsym,stmaryrd,multicol}

\usepackage[usenames]{color}
\usepackage[all]{xy}

\usepackage[top=25truemm,bottom=25truemm,left=25truemm,right=25truemm]{geometry}
\newtheorem{thm}{Theorem}[section]
\newtheorem{lem}[thm]{Lemma}
\newtheorem{prop}[thm]{Proposition}
\newtheorem{cor}[thm]{Corollary}
\theoremstyle{definition}
\newtheorem{dfn}[thm]{Definition}

\newtheorem{ex}[thm]{Example}

\newtheorem{rmk}[thm]{Remark}
\theoremstyle{remark}

\newtheorem*{ac}{Acknowledgments}

\numberwithin{equation}{thm}
\def\Add{\operatorname{Add}}

\def\C{\mathrm{C}}

\def\Cok{\operatorname{Coker}}

\def\Cone{\operatorname{Cone}}

\def\Ext{\operatorname{Ext}}

\def\Gdim{\operatorname{G-dim}}
\def\ge{\geqslant}
\def\gp{\operatorname{GP}}

\def\H{\mathrm{H}}
\def\HL{\operatorname{H}}

\def\Hom{\operatorname{Hom}}

\def\Id{\operatorname{Id}}
\def\image{\operatorname{Im}}

\def\K{\mathcal{K}}
\def\Ker{\operatorname{Ker}}

\def\le{\leqslant}
\def\lhom{\operatorname{\underline{Hom}}}

\def\Mod{\operatorname{Mod}}
\def\mod{\operatorname{mod}}

\def\pd{\operatorname{pd}}

\def\proj{\operatorname{proj}}

\def\syz{\Omega}
\def\t{\mathrm{T}}

\def\Tor{\operatorname{Tor}}
\def\tr{\operatorname{Tr}}

\def\X{\mathcal{X}}

\def\ZZ{\mathbb{Z}}
\begin{document}
\allowdisplaybreaks
\title[On the Auslander--Bridger--Yoshino theory for complexes]{On the Auslander--Bridger--Yoshino theory for complexes of \\finitely generated projective modules}
\author{Yuya Otake}
\address{Graduate School of Mathematics, Nagoya University, Furocho, Chikusaku, Nagoya 464-8602, Japan}
\email{m21012v@math.nagoya-u.ac.jp}

\thanks{2020 {\em Mathematics Subject Classification.} 16E05, 16E10, 13D02.}
\thanks{{\em Key words and phrases.} homotopy category, stable category, syzygy, cosyzygy, $n$-torsionfree module, Gorenstein projective module, Gorenstein dimension.}
\thanks{The author was partly supported by Grant-in-Aid for JSPS Fellows 23KJ1119.}
\begin{abstract}
Let $R$ be a two-sided noetherian ring. Auslander and Bridger developed a theory of projective stabilization of the category of finitely generated $R$-modules, which is called the stable module theory. Recently, Yoshino established a stable ``complex'' theory, i.e., a theory of a certain stabilization of the homotopy category of complexes of finitely generated projective $R$-modules. We introduce higher versions of several notions introduced by Yoshino, such as ${}^\ast$torsionfreeness and ${}^\ast$reflexivity. %for complexes.
Also, we prove the Auslander--Bridger approximation theorem for complexes of finitely generated projective $R$-modules.
%Let $R$ be a commutative Noetherian ring. Auslander and Bridger established the theory of the stable category of the category of finitely generated $R$-modules.
%Yoshino extends their theory to the homotopy category of complexes of finitely generated projective $R$-modules. In this paper, we give some refinements and generalizations of Yoshino’s theory.
\end{abstract}
\maketitle
\tableofcontents
%\tableofcontents
%%%%%%%%%%%%%%%%%%%%%%%%%%%%%%%%%%%%%%%%%%%%%%%%%%%%%%%%%%%%
\section{Introduction}
Let $R$ be a two-sided noetherian ring.
Denote by $\mod R$ the category of finitely generated (right) $R$-modules and by $\proj R$ the full subcategory of $\mod R$ consisting of finitely generated projective $R$-modules.
Auslander and Bridger \cite{AB} established a theory of the stable category $\underline{\mod R} $, which is the ideal quotient of $\mod R$ by $\proj R$.
The main subject of the stable module theory is the behavior of the syzygy functor $\syz : \underline{\mod R}\to\underline{\mod R}$ and the transpose functor $\tr : \underline{\mod R}\to \underline{\mod R^{\mathrm{op}}}$.
By exploring these functors, Auslander and Bridger introduced the notions of $n$-torsionfree modules and Gorenstein dimension, and studied them deeply.
Recently, Yoshino \cite{Yos} found out that an analogous theory of the Auslander--Bridger theory holds for a certain stabilization of the homotopy category of complexes of finitely generated projective $R$-modules.
% Recently, Yoshino \cite{Yos} developed an analogous theory of Auslander--Bridger theory  for a certain stabilization of the homotopy category of complexes of finitely generated projective $R$-modules.
%His theory is called the stable ``complex'' theory.
We denote by $\K(\proj R)$ the homotopy category of complexes of finitely generated projective $R$-modules and by $\Add R$ the smallest full subcategory of $\K(\proj R)$ that contains $R$ and is closed under shifts and direct summands, and possibly infinite coproducts.
Yoshino's theory, which is called the stable complex theory, is about the ideal quotient $\underline{\K(\proj R)}$ of $\K(\proj R)$ by $\Add R$.
For an object $X$ of $\K(\proj R)$ the syzygy $\syz X$ of $X$ is defined as the cocone of a right $\Add R$-approximation $p : P\to X$ of $X$.
So there exists a triangle
$$
\syz X\to P\to X \to \syz X[1]
$$
in $\K(\proj R)$.
Also, the $R$-dual functor $(-)^\ast =\Hom_R(-,R):\K(\proj R)\to\K(\proj R^{\mathrm{op}})$ plays a role of the transpose functor, and the cosyzygy functor is defined as the composition $\syz^{-}=(-)^\ast \circ \syz \circ (-)^\ast : \K(\proj R)\to\K(\proj R)$.
Yoshino showed that the pair $(\syz^{-1}, \syz)$ of functors from $\K(\proj R)$ to itself is an adjoint pair; there is a functorial isomorphism
$$
\Hom_{\underline{\K(\proj R)}}(\Omega^- X, Y)\cong\Hom_{\underline{\K(\proj R)}}(X, \Omega Y)
$$
for all $X, Y\in\underline{\K(\proj R)}$.

An important tool in the stable module theory of Auslander and Bridger is the general theory of functors between abelian categories developed in \cite{AusC} and \cite[Chapter 1]{AB}.
In particular, for a finitely generated $R$-module $M$ and its evaluation map $\varphi_M : M\to M^{\ast\ast}$, the exact sequence called the {\em Auslander--Bridger sequence}
% defined by $\varphi_M(x)(f)=f(x)$ for $x\in M$ and $f\in M^\ast$, the exact sequence called the {\em Auslander--Bridger sequence}
\begin{equation}\label{clasABseq}
0\to \Ext^1_{R^{\mathrm{op}}}(\tr M, R)\to M\xrightarrow{\varphi_M}M^{\ast\ast}\to \Ext^2_{R^{\mathrm{op}}}(\tr M,R)\to 0
\end{equation}
plays an essential role.
Yoshino showed the following exact sequence for all complexes $X$ of finitely generated projective $R$-modules and all integers $i$:
\begin{equation}\label{YosABseq}
0\to\Ext^1_R(C^{i+1}(X),R)\to \H^{-i}(X^\ast)\xrightarrow{\rho^i_X} \H^i(X)^\ast \to\Ext^2_R(C^{i+1}(X),R),
\end{equation}
where the $R^{\mathrm{op}}$-homomorphism $\rho^i_X : \H^{-i}(X^\ast)\to \H^i(X)^\ast $ is defined by $\rho^i_X(\overline{f})(\overline{x})=f(x)$ for $\overline{f}\in \H^{-i}(X^\ast)$ and $x\in \H^i(X)$, $C^{i+1}(X)$ is the cokernel of the $i$th differential map of $X$.
% and $\H^{i}(X)$ is the $i$th cohomology of $X$.
This exact sequence partially recovers the Auslander--Bridger sequence (\ref{clasABseq}).
In terms of the homomorphism $\rho^i_X$, ${}^\ast$torsionfreeness and ${}^\ast$reflexivity are defined for objects of $\K(\proj R)$; an object $X$ of $\K(\proj R)$ is called {\em ${}^\ast$torsionfree} if the homomorphism $\rho^i_X : \H^{-i}(X^\ast)\to \H^i(X)^\ast $ is injective, and called {\em ${}^\ast$reflexive} if  $\rho^i_X$ is bijective.
Yoshino studied ${}^\ast$torsionfreeness and ${}^\ast$reflexivity of syzygies and cosyzygies.

The purpose of this paper is to give refinements and generalizations of Yoshino's results.
First, we construct a long exact sequence of Auslander--Bridger type for right or left exact functors between abelian categories.
As a special case of the long exact sequence of Auslander--Bridger type, we complement the exact sequence (\ref{YosABseq}).
Based on this exact sequence, for each integer $n$ we introduce $n$-torsionfreeness for objects of $\K(\proj R)$.
This notion is a higher version of the dual concepts of ${}^\ast$torsionfreeness and ${}^\ast$reflexivity for objects of $\K(\proj R)$.
Auslander and Bridger studied when syzygy modules become $n$-torsionfree.
Thus, it is a natural question to ask when syzygies of objects of $\K(\proj R)$ become $n$-torsionfree.
The key in stating our results is the concept of cohomologically ghost triangles: a triangle $X\to Y\to Z\to X[1]$ in $\K(\proj R)$ is called {\em cohomologically ghost} if the sequence $0\to \H^i(X)\to \H^i(Y)\to \H^i(Z)\to0$ of cohomologies of complexes is exact for all integers $i$.
In our framework, cohomologically ghost triangles play the role of short exact sequences in the module category $\mod R$.
Using this terminology, we can state our first main result.
The equivalences $(1)\Leftrightarrow(2)$ and $(1)\Leftrightarrow(3)$ of the following theorem are analogues of \cite[Theorem 2.17]{AB} and \cite[Proposition 2.21]{AB}, respectively.

\begin{thm}[Theorems \ref{ntrfeq} and \ref{ABapp}]\label{yosintro1}
Let $X$ be an object of $\K(\proj R)$ and $n\ge0$ an integer.
Then the following are equivalent.
\begin{enumerate}[\rm(1)]
   \item
   The syzygy $\syz^n X$ is $n$-torsionfree.
   \item
   For each $0\le i\le n-1$, there exists a cohomologically ghost triangle $X_i\to P_i\to X_{i+1}\to X_i[1]$ in $\K(\proj R)$ with $P_i\in\Add R$ whose $R$-dual $X_{i+1}^\ast\to P_i^\ast\to X_i^\ast\to X_{i+1}^\ast[1]$ is also cohomologically ghost, such that $\syz^n X\cong X_0$ in $\K(\proj R)$.
   %There exist cohomologically ghost triangles $X_i\to P_i\to X_{i+1}\to X_i[1]$ in $\K(\proj R)$ for $0\le i\le n-1$ such that $\syz^n X\cong X_0$ in $\K(\proj R)$ and the $R$-dual $X_{i+1}^\ast \to P_i^\ast \to X_i^\ast \to X_{i+1}^\ast[1]$ is also cohomologically ghost for all $0\le i\le n-1$.
   \item
There exists a cohomologically ghost triangle $Y\to G\to X\to Y[1]$ in $\K(\proj R)$ such that $G^\ast$ is $n$-torsionfree and $Y$ has projective dimension at most $n-1$.
\end{enumerate}
\end{thm}

Here, projective dimension for objects of $\K(\proj R)$ is defined as the infimum of lengths of right $\Add R$-resolutions, as an analogue of usual projective dimension for modules; see Definition \ref{defGdim} for details.

The notions of Gorenstein projective modules and Gorenstein dimension were introduced by Auslander and Bridger \cite{AB}.
A finitely generated $R$-module $M$ is called {\em Gorenstein projective} if both $M$ and $\tr M$ are $\infty$-torsionfree.
As an analogy, we introduce Gorenstein projectivity for objects of $\K(\proj R)$, that is, an object $X$ of $\K(\proj R)$ is called {\em Gorenstein projective} if both $X$ and $X^\ast$ are $\infty$-torsionfree in our sense.
%Also, the Gorenstein dimension for objects of $\K(\proj R)$ is defined as the infimum of the lengths of approximations by Gorenstein 
Also, Gorenstein dimension for objects of $\K(\proj R)$ is defined in the same way as projective dimension.
%; see Definition \ref{defGdim} for details.
Here, it must be noted that projective dimension and Gorenstein dimension introduced in this paper are quite different from those defined for homologically bounded complexes by Foxby \cite{Fox1} and Yassemi \cite{Yassemi}; see Example \ref{ntrfex}.
%The projective dimension for homologically bounded complexes of $R$-modules
%Foxby \cite{Fox1} introduced the projective dimension for homologically bounded complexes of $R$-modules, which are not necessarily consisting of fintely generated projective modules.
The next theorem is the main result for Gorenstein projective complexes and Gorenstein dimension in our sense.

\begin{thm}[Theorems \ref{Gdimca}, \ref{SWet} and Corollary \ref{AusBuchapp}]\label{yosintro2}
Let $X$ be an object of $\K(\proj R)$.
Then the following hold.
\begin{enumerate}[\rm(1)]
   \item
   The following are equivalent.
   \begin{itemize}
   \item[(1-a)] The complex $X$ is Gorenstein projective.
   \item[(1-b)] For each integer $t$, there exists a cohomologically ghost triangle $X_{t+1}\to P_t\to X_t\to X_{t+1}[1]$ in $\K(\proj R)$ with $P_t\in\Add R$ whose $R$-dual $X_{t}^\ast\to P_t^\ast\to X_{t+1}^\ast\to X_{t}^\ast[1]$ is also cohomologically ghost, such that $X\cong X_0$ in $\K(\proj R)$.
   %There exist cohomologically ghost triangles $X_{t+1}\to P_t\to X_t\to X_{t+1}[1]$ in $\K(\proj R)$ with $P_t\in\Add R$ for all integers $t$ such that $X\cong X_0$ in $\K(\proj R)$ and the $R$-dual triangle $X_t^\ast \to P_t^\ast \to X_{t+1}^\ast \to X_t^\ast[1]$ is also cohomologically ghost for all integers $t$.
   \item[(1-c)] For each integer $t$, there exists a cohomologically ghost triangle $X_{t+1}\to G_t\to X_t\to X_{t+1}[1]$ in $\K(\proj R)$ with $G_t$ Gorenstein projective whose $R$-dual $X_{t}^\ast\to P_t^\ast\to X_{t+1}^\ast\to X_{t}^\ast[1]$ is also cohomologically ghost, such that $X\cong X_0$ in $\K(\proj R)$.
   %There exist cohomologically ghost triangles $X_{t+1}\to G_t\to X_t\to X_{t+1}[1]$ in $\K(\proj R)$ with $G_t$ Gorenstein projective for all integers $t$ such that $X\cong X_0$ in $\K(\proj R)$ and the $R$-dual triangle $X_t^\ast \to G_t^\ast \to X_{t+1}^\ast \to X_t^\ast[1]$ is also cohomologically ghost for all integers $t$.
   \end{itemize}
   \item
   The following are equivalent.
   \begin{itemize}
   \item[(2-a)] The complex $X$ has finite Gorenstein dimension.
   \item[(2-b)] There exists a positive integer $n$ such that the module $C^i(X)$ has Gorenstein dimension at most $n$ for all integers $i$.
   %The supremum of Gorenstein dimensions of the modules $C^i(X)$ for 
   %One has $\sup\{\Gdim_R C^i(X) \mid i\in\ZZ\}<\infty$.
   \item[(2-c)] There exists a cohomologically ghost triangle $Y\to G\to X\to Y[1]$ in $\K(\proj R)$ such that $Y$ has finite projective dimension and $G$ is Gorenstein projective.
   \end{itemize}
\end{enumerate}
\end{thm}

The equivalence (1-a) $\Leftrightarrow $ (1-b) in the above theorem is an analogue of the well-known equivalence between having a complete resolution and being Gorenstein projective, and the equivalence (1-a) $\Leftrightarrow$ (1-c) is an analogue of \cite[Theorem A]{SSW} for finitely generated modules.
Also, the module version of the equivalence (2-a) $\Leftrightarrow$ (2-c) is none other than a consequence of Auslander--Buchweitz theory \cite{ABu}.

The organization of this paper is as follows.
In Section 2, we state several notions and their basic properties for later use.
In Section 3, we construct long exact sequences of Auslander--Bridger type for functors between general abelian categories and give some examples in the case of the module category.
In Section 4, we recall the basic properties of the full subcategory $\Add R$ of $\K(\proj R)$ and introduce torsionlessness and reflexivity for objects of $\K(\proj R)$.
In the final Section 5, we define $n$-torsionfreeness as a higher version of torsionlessness and reflexivity and prove Theorems \ref{yosintro1} and \ref{yosintro2}.

%%%%%%%%%%%%%%%%%%%%%%%%%%%%%%%%%%%%%%%%%%%%%%%%%%%%%%%%%%%%
\section{Preliminaries}
Throughout this paper, let $R$ be a two-sided noetherian ring.
%All subcategories are assumed strictly full.
We denote by $\Mod R$ the category of (right) $R$-modules and by $\mod R$ the category of finitely generated (right) $R$-modules.
In this paper, we consider the stable complex theory, which is a complex version of the stable module theory developed by Auslander and Bridger \cite{AB}.
So we recall the notions about the stable category $\underline{\mod R}$ of $\mod R$, which play a central role in the Auslander--Bridger theory.

\begin{dfn}
\begin{enumerate}[\rm(1)]
   \item
   We denote by $\underline{\mod R}$ the {\em stable category} of $\mod R$.
   The objects of $\underline{\mod R}$ are the same as those of $\mod R$.
   The morphism set of objects $X, Y$ of $\underline{\mod R}$ is defined by
   $$
   \Hom_{\underline{\mod R}}(X,Y)
   =\lhom_R(X,Y)
   =\Hom_R(X,Y)/{\mathcal{P}(X,Y)},
   $$
   where $\mathcal{P}(X,Y)$ is the set of $R$-morphisms $X\to Y$ factoring through finitely generated projective modules.
   %For any morphism $f:X\to Y$, we denote by $\underline{f}$ the image of $f$ in $\lhom_R(X,Y)$.
   \item
   Let $M$ be a finitely generated $R$-module and $P_1\xrightarrow{d_1}P_0\xrightarrow{d_0}M\to0$ a finite projective presentation of $M$.
   The {\em syzygy} $\syz M$ of $M$ is defined as $\image d_1$.
   Note that $\syz M$ is uniquely determined by $M$ up to projective summands.
   Taking the syzygy induces an additive functor $\syz:\underline{\mod R}\to\underline{\mod R}$.
   Inductively, we define $\syz^n=\syz\circ\syz^{n-1}$ for an integer $n>0$.
   The {\em (Auslander) transpose} $\tr M$ of $M$ is defined as $\Cok d_1^\ast$.
   Note that $\tr M$ is uniquely determined by $M$ up to projective summands.
   Taking the transpose induces an additive functor $\tr:\underline{\mod R}\to\underline{\mod R^{\mathrm{op}}}$.
   \item
   A finitely generated $R$-module $M$ is called {\em $n$-torsionfree} if $\Ext^i_{R^{\mathrm{op}}}(\tr M,R)=0$ for all $1\le i\le n$.
   %We denote by $\tf_n(R)$ the subcategory of $\mod R$ consisting of $n$-torsionfree modules
   Also, we say that a finitely generated $R$-module $M$ is {\em Gorenstein projective} if $M$ and $\tr$ are $n$-torsionfree for all $n\ge 1$, that is, $\Ext^n_R(M,R)=0$ and $\Ext^n_{R^{\mathrm{op}}}(\tr M,R)=0$ for all $n\ge 1$.
   \item 
   We denote by $(-)^\ast$ the $R$-dual functor $\Hom_R(-,R)$.
   For an $R$-module $M$ we denote by $\varphi_M : M\to M^{\ast\ast}$ the canonical map given by $\varphi_M(x)(f)=f(x)$ for $x\in M$ and $f\in M^\ast$.
   A finitely generated $R$-module $M$ is called {\em torsionless} if the canonical map $\varphi_M$ is injective and called {\em reflexive} if $\varphi_M$ is bijective.
   Note that a finitely generated $R$-module $M$ is torsionless if and only if it is $1$-torsionfree.
   Similarly, a finitely generated $R$-module $M$ is reflexive if and only if it is $2$-torsionfree.
   %We denote by φM : M → M∗∗ the canonical map given by φM(x)(f) = f(x)
   %for x ∈ M and f ∈ M∗
   \item
   %Let $m,n\in\ZZ_{\ge0}\cup\{\infty\}$.
   % We denote by $\G_{m,n}(R)$, or simply $\G_{m,n}$, the subcategory of $\mod R$ consisting of $R$-modules $M$ such that $\Ext^i_R(M,R)=0$ for all $1\le i\le m$ and $\Ext^j_{R^{\mathrm{op}}}(\tr M,R)=0$ for all $1\le j\le n$.
    We denote by $\proj R$ (resp. $\gp R$) the full subcategory of $\mod R$ consisting of finitely generated projective (resp. Gorenstein projective) $R$-modules.
    %Note that $\gp(R)=\G_{\infty, \infty}$.
    %A finitely generated $R$-module $M$ is called {\em $n$-torsionfree} if $M$ belongs to $\G_{0,n}$.
    %, that is, we set $\tf_n(R)=\G_{0,n}$.
    \item
    The {\em projective dimension} (resp. {\em Gorenstein dimension}) of a finitely generated $R$-module $M$ is defined to be the infimum of integers $n$ such that there exists an exact sequence
    $$
    0\to X_n\to X_{n-1}\to\cdots\to X_1\to X_0\to M\to0
    $$
    of finitely generated $R$-modules with $X_i$ projective (resp. Gorenstein projective).
    \item 
    Let $M, N$ be finitely generated $R$-modules.
    We say that $M$ and $N$ are {\em stably isomorphic} if there are finitely generated projective modules $P, Q$ such that $M\oplus P\cong N\oplus Q$, and then write $M\approx N$.
    Note that $M\approx N$ if and only if $M$ and $N$ are isomorphic as objects of $\underline{\mod R}$.
\end{enumerate}
\end{dfn}

We fix some notions for complexes.
Let $\mathcal{A}$ be an abelian category.
A complex $X=X^{\bullet}$ of objecsts of $\mathcal{A}$ is denoted as
$$
X^{\bullet}=(\cdots\xrightarrow{d_X^{i-2}} X^{i-1}\xrightarrow{d_X^{i-1}}X^i\xrightarrow{d_X^{i}}X^{i+1}\xrightarrow{d_X^{i+1}}\cdots).
$$
We denote by $C^i(X)$ the cokernel of the differential $d_X^{i-1}:X^{i-1}\to X^i$.
Similarly, for each integer $i$, we define the $i$th {\em cocycle} $Z^i(X)$ to be the kernel of $d_X^i$, the $i$th {\em coboundary} $B^i(X)$ to be the image of $d_X^{i-1}$ and the $i$th {\em cohomology} $\H^i(X)={Z^i(X)}/{B^i(X)}$.
For an integer $n$, the complex $X[n]$ is defined by $X[n]^i=X^{n+i}$ and $d_{X[n]}^i = (-1)^n d_X^{n+i}$ for all $i$.
%For an integer j, the complex X[j] is defined by X[j]i = Xi−j and ∂X[j]i = (−1)j∂Xi−jfor all i.

%Similarly, the $i$-th cocycle $Z^i(X)$ is the kernel of the map $d_X^i$ and the $i$-th coboundary $B^i(X)$ is the image of $d_X^{i-1}$.

%For each integer i, we define the ith cycle Zi(X) = Ker ∂i, the ith boundary Bi(X) = Im ∂i+1 and the ith homology Hi(X) = Zi(X)/ Bi(X)
Let $f:X\to Y$ be a homomorphism of complexes of objects of $\mathcal{A}$.
The {\em mapping cone}, or simply cone, of $f$ is the complex defined as
$$
\Cone f = (\cdots\to X^i \oplus Y^{i-1}\xrightarrow{\begin{pmatrix}-d_X^i & 0 \\ f^i & d_Y^{i-1}\end{pmatrix}}X^{i+1}\oplus Y^i\xrightarrow{\begin{pmatrix}-d_X^{i+1} & 0 \\ f^{i+1} & d_Y^{i}\end{pmatrix}}X^{i+2}\oplus Y^{i+1}\to\cdots).
$$
We often use the symbol $^\bullet$ for claims on complexes to omit the phrase “for all $i\in\ZZ$”.
For example, “$C^{\bullet}(X)$ is projective” means that $C^i(X)$ is projective for all integers $i$, and “$\H^{\bullet}(f):\H^{\bullet}(X)\to \H^{\bullet}(Y)$ is surjective” means that $\H^i(f):\H^i(X)\to \H^i(Y)$ is surjective for all integers $i$.
%We denote by $\K(R)=\K(\proj R)$ the homotopy category consisting of complexes of finitely generated projective $R$-modules.

%%%%%%%%%%%%%%%%%%%%%%%%%%%%%%%%%%%%%%%%%%%%%%%%%%%%%%%%%%%%
\section{Long exact sequences of Auslander--Bridger type}
The two exact sequences described in \cite[Theorem 2.8]{AB}, which are called {\em Auslander--Bridger sequences}, are fundamental tools in the stable module theory.
The existence of these sequences was found in \cite{AusC}.
Also, \cite[Theorem 2.3]{Yos} is the Auslander-Bridger sequence for complexes consisting of projective modules.
In this section, in a more general setting, we describe the existence of long exact sequences of Auslander--Bridger type for right or left exact functors and complexes.

%Let $\mathcal{A}$, $\mathcal{A}^\prime$ and $\mathcal{B}$ be abelian categories.
Let $\mathcal{A}$ and $\mathcal{B}$ be abelian categories.
%Suppose that $\mathcal{A}$ has enough projective objects and $\mathcal{A'}$ has enough injective objects.
Suppose that $\mathcal{A}$ has enough projective objects. % and $\mathcal{A'}$ has enough injective objects.
%For a left exact additive contravariant (resp. covariant) functor $F:\mathcal{A}\to \mathcal{B}$ (resp. $F':\mathcal{A'}\to \mathcal{B}$), we denote by $R^n F : \mathcal{A}\to\mathcal{B}$ (resp. $R^n F' : \mathcal{A'}\to\mathcal{B}$) the $n$th right derived functor of $F$ (resp. $F'$).
For a left (resp. right) exact additive contravariant (resp. covariant) functor $F:\mathcal{A}\to \mathcal{B}$ (resp. $G:\mathcal{A}\to \mathcal{B}$), we denote by $R^n F : \mathcal{A}\to\mathcal{B}$ (resp. $L^n G : \mathcal{A}\to\mathcal{B}$) the $n$th right (resp. left) derived functor of $F$ (resp. $G$).
%Similarly, for a right exact additive covariant (resp. contravariant) functor $G:\mathcal{A}\to \mathcal{B}$ (resp. $G':\mathcal{A'}\to \mathcal{B}$), we denote by $R^n G : \mathcal{A}\to\mathcal{B}$ (resp. $R^n G' : \mathcal{A'}\to\mathcal{B}$) the $n$-th left derived functor of $G$ (resp. $G'$).
%Similarly, for a right exact additive covariant (resp. contravariant) functor $G:\mathcal{A}\to \mathcal{B}$ (resp. $G':\mathcal{A'}\to \mathcal{B}$), we denote by $R^n G : \mathcal{A}\to\mathcal{B}$ (resp. $R^n G' : \mathcal{A'}\to\mathcal{B}$) the $n$-th left derived functor of $G$ (resp. $G'$).

Let $F:\mathcal{A}\to\mathcal{B}$ (resp. $G:\mathcal{A}\to\mathcal{B}$) be a left (resp. right) exact additive contravariant (resp. covariant) functor and $X$ a complex of objects of $\mathcal{A}$.
Then, for any integer $i$, the natural map $\rho^i_{X,F}: \H^{-i}(F(X))\to F(\H^i(X))$ (resp. $\tau^i_{X,G}:G(\H^i(X))\to \H^i(G(X))$) is defined; see \cite[Chapter IV]{CE}.
If each component of $X$ is orthogonal to the $i$th derived functor of $F$ (resp. $G$) for all $i>0$, then the following exact sequences exist.
These sequences can be regarded as a generalization of Auslander--Bridger sequences; see Example \ref{ABex}.

\begin{thm}\label{ABtype}
%Let $\mathcal{A}$, $\mathcal{A}^\prime$ and $\mathcal{B}$ be abelian categories.
%Suppose that $\mathcal{A}$ has enough projective objects and $\mathcal{A'}$ has enough injective objects.
Let $\mathcal{A}$ and $\mathcal{B}$ be abelian categories.
Suppose that $\mathcal{A}$ has enough projective objects.
\begin{enumerate}[\rm(1)]
   \item 
   Let $F:\mathcal{A}\to \mathcal{B}$ be a left exact additive contravariant functor and $X$ a complex of objects of $\mathcal{A}$.
   Suppose that $R^n F(X^j)$ for all $n>0$ and $j\in\ZZ$.
   For any integer $i$ there exists a long exact sequence
   $$
   \xymatrix@R-1pc@C-1pc{
   0\ar[r]&R^1 F(C^{i+1}(X))\ar[r]&\H^{-i}(F(X))\ar[r]^-{\rho^i_{X,F}}&F(\H^i(X))\\
   \ar[r]&R^2 F(C^{i+1}(X))\ar[r]&R^1 F(C^i(X))\ar[r]&R^1 F(\H^i(X))\\
   \ar[r]&R^3 F(C^{i+1}(X))\ar[r]&R^2 F(C^i(X))\ar[r]&R^2 F(\H^i(X))\ar[r]&\cdots.
   }
   $$
   %\item
   %Let $F':\mathcal{A'}\to \mathcal{B}$ be a left exact additive covariant functor and $Y$ a complex of objects of $\mathcal{A'}$.
   %Suppose that $R^n F'(Y^j)$ for all $n>0$ and $j\in\ZZ$.
   %$$
   %\xymatrix@R-1pc@C-1pc{
   %0\ar[r]&R^1 F'(Z^{i-1}(Y))\ar[r]&\H^{i}(F(Y))\ar[r]^-{\delta^i_{M,Y}}&F(\H^{i}(Y))\\
   %\ar[r]&R^2 F'(Z^{i-1}(Y))\ar[r]&R^1 F'(Z^{i}(Y))\ar[r]&R^1 F'(\H^{i}(Y))\\
   %\ar[r]&R^3 F'(Z^{i-1}(Y))\ar[r]&R^2 F'(Z^{i}(Y))\ar[r]&R^2 F'(\H^{i}(Y))\ar[r]&\cdots.
   %}
   %$$
   \item 
   Let $G:\mathcal{A}\to \mathcal{B}$ be a right exact additive covariant functor and $X$ a complex of objects of $\mathcal{A}$.
   Suppose that $L^n G(X^j)$ for all $n>0$ and $j\in\ZZ$.
   For any integer $i$ there exists a long exact sequence
   $$
   \xymatrix@R-1pc@C-1pc{
   \cdots\ar[r]&L^2 G(\H^i(X))\ar[r]&L^2 G(C^i(X))\ar[r]&L^3 G(C^{i+1}(X))\\
   \ar[r]&L^1 G(\H^i(X))\ar[r]&L^1 G(C^i(X))\ar[r]&L^2 G(C^{i+1}(X))\\
   \ar[r]&G(\H^i(X))\ar[r]^-{\tau^i_{X,G}}&\H^i(G(X))\ar[r]&L^1 G(C^{i+1}(X))\ar[r]&0.
   }
   $$
   \end{enumerate}
\end{thm}

\begin{proof}
%We first note that the module $B^{i+1}(X)$ is a first syzygy of $C^{i+1}(X)$ for all $i\in\ZZ$ by the exact sequence $0\to B^{i+1}(X)\to X^{i+1}\to C^{i+1}(X)\to0$.
We first note that $R^k F(B^{i+1}(X))\cong R^{k+1} F(C^{i+1}(X))$ for all $k>0$ and $i\in\ZZ$ by the exact sequence $0\to B^{i+1}(X)\to X^{i+1}\to C^{i+1}(X)\to0$ and the assumption that $R^k F(X^{i+1})=0$.
From the exact sequence $0\to \H^i(X)\to C^i(X)\to B^{i+1}(X)\to0$ we obtain an exact sequence
\begin{equation}\label{1.1F}
\begin{split}
0\to F(B^{i+1}(X))\to F(C^i(X))\to F(\H^i(X))\to
R^2 F(C^{i+1}(X))\to R^1 F(C^i(X))\to R^1 F(\H^i(X))\to\cdots.
\end{split}
\end{equation}

Also, by the exact sequence $F(X^{i+1})\to F(B^{i+1}(X))\to R^1 F(C^{i+1}(X))\to0$ and the commutative diagram
$$
\xymatrix@R-1pc@C-1pc{
F(X^{i+1})\ar[rr]\ar[dd]&&F(X^i)\ar@{=}[dd]\\
&&\\
F(B^{i+1}(X))\ar@<-0.3ex>@{^{(}->}[rr]&&F(X^i),
}
$$
we have the exact sequence
\begin{equation}\label{1.2F}
0\to B^{-i}(F(X))\to F(B^{i+1}(X))\to R^1 F(C^{i+1}(X))\to0.
\end{equation}
Combining these exact sequences, we get the commutative diagram
$$
\xymatrix@R-1pc@C-1pc{
&0\ar[d]&&&&\\
0\ar[r]&B^{-i}(F(X))\ar[r]\ar[d]&Z^{-i}(F(X))\ar[r]\ar[d]&\H^{-i}(F(X))\ar[r]\ar[d]^-{\rho^i_{X, F}}&0&\\
0\ar[r]&F(B^{i+1}(X))\ar[r]\ar[d]&F(C^{i}(X))\ar[r]&F(\H^i(X))\ar[r]&R^2 F(C^{i+1}(X))\ar[r]&\cdots\\
&R^1 F(C^{i+1}(X))\ar[d]&&&&\\
&0&&&&
}
$$
with exact rows and columns, where the second row is the exact sequence (\ref{1.1F}), the first column is (\ref{1.2F}) and the second column is the isomorphism $Z^{-i}(F(X))\cong F(C^{i}(X))$ obtained from the exact sequence $0\to F(C^i(X))\to F(X^i)\to F(X^{i+1})$.
The desired exact sequence is obtained from the snake lemma.

The proof of the assertion (2) is similar to the proof of the assertion (1).
For the reader’s convenience, we give the proof of it.
%For the exact sequence
We obtain an exact sequence
\begin{equation}\label{2.1G}
\begin{split}
\cdots\to L^1 G(\H^i(X))\to L^1 G(C^i(X))\to L^2 G(C^{i+1}(X))
\to G(\H^i(X)) \to G(C^i(X)) \to G(B^{i+1}(X))\to0.
\end{split}
\end{equation}
Also, by the exact sequence $0\to L^1 G(C^{i+1}(X))\to G(B^{i+1}(X))\to G(X^{i+1})$ and the commutative diagram
$$
\xymatrix@R-1pc@C-1pc{
G(X^i)\ar[rr]\ar@{=}[dd]\ar@{->>}[rr]&&G(B^{i+1}(X))\ar[dd]\\
&&\\
G(X^i)\ar[rr]&&G(X^{i+1}),
}
$$
we have the exact sequence
\begin{equation}\label{2.2G}
0\to L^1 G(C^{i+1}(X))\to G(B^{i+1}(X))\to B^{i+1}(G(X))\to0.
\end{equation}
Combining these exact sequences, we get the commutative diagram
$$
\xymatrix@R-1pc@C-1pc{
&&&&0\ar[d]&\\
&&&&L^1 G(C^{i+1}(X))\ar[d]&\\
\cdots\ar[r]& L^2 G(C^{i+1}(X))\ar[r]&G(\H^i(X))\ar[r]\ar[d]^-{\tau^i_{X,G}}&G(C^{i}(X))\ar[r]\ar[d]&G(B^{i+1}(X))\ar[r]\ar[d]&0\\
&0\ar[r]&\H^i(G(X))\ar[r]&C^i(G(X))\ar[r]&B^{i+1}(G(X))\ar[r]\ar[d]&0\\
&&&&0&
}
$$
with exact rows and columns, where the first row is the exact sequence (\ref{2.1G}), the first column is (\ref{2.2G}) and the second column is the isomorphism $G(C^i(X))\cong C^i(G(X))$ obtained from the exact sequence $G(X^{i-1})\to G(X^i)\to G(C^i(X))\to0$.
Similar to the proof of the assertion (1), we obtain the desired exact sequence by applying the snake lemma.
\end{proof}

\begin{rmk}\label{ABtypeZ}
Let $\mathcal{A'}$ and $\mathcal{B}$ be abelian categories and $F':\mathcal{A'}\to\mathcal{B}$ be a left exact additive covariant functor.
Suppose that $\mathcal{A'}$ has enough injective objects.
Then from Theorem \ref{ABtype} we obtain the following long exact sequence:
Let $Y$ be a complex of objects of $\mathcal{A'}$ and $i$ an integer.
Suppose that $R^n F'(Y^j)$ for all $n>0$ and $j\in\ZZ$.
Then there exists a long exact sequence
$$
\xymatrix@R-1pc@C-1pc{
0\ar[r]&R^1 F'(Z^{i-1}(Y))\ar[r]&\H^{i}(F'(Y))\ar[r]^-{\delta^i_{F',Y}}&F'(\H^{i}(Y))\\
\ar[r]&R^2 F'(Z^{i-1}(Y))\ar[r]&R^1 F'(Z^{i}(Y))\ar[r]&R^1 F'(\H^{i}(Y))\\
\ar[r]&R^3 F'(Z^{i-1}(Y))\ar[r]&R^2 F'(Z^{i}(Y))\ar[r]&R^2 F'(\H^{i}(Y))\ar[r]&\cdots.
}
$$
Similarly as above, for a right exact additive contravariant $G':\mathcal{A'}\to \mathcal{B}$, we can state the corresponding statement.
\end{rmk}

The result below is a direct corollary of the above theorem and remark.
The assertion (1) is the long version of \cite[Theorem 2.3]{Yos}.
The assertions (2) and (3) are the $\Tor$ version and the dual version of (1), respectively.
Also, we can state the same claim for the local cohomology functor, when $R$ is commutative.
Suppose that $R$ is commutative and let $\mathfrak{a}$ be an ideal of $R$.
Denote by $\Gamma_{\mathfrak{a}}(-)$ the {\em $\mathfrak{a}$-torsion functor}; recall that $\Gamma_{\mathfrak{a}}(M)=\{x\in M\mid \mathfrak{a}^r x=0\text{ for some }r>0\}$ for an $R$-module $M$.
Let $\HL^i_{\mathfrak{a}}(-)$ be the $i$th {\em local cohomology functor with respect to $\mathfrak{a}$}, that is, the $i$th right derived functor of $\Gamma_{\mathfrak{a}}(-)$; see \cite{BS, BH} for details.

\begin{cor}\label{fundex}
%Let $X\in \K(R)$ and $M\in\Mod R$.
Let $R$ be a two-sided noetherian ring.
Let $M\in\Mod R$ and $N\in\Mod R^{\mathrm{op}}$.
\begin{enumerate}[\rm(1)]
   \item 
   %For any integer $i$ there exists a long exact sequence
   Let $X$ be a complex of projective $R$-modules.
   Then for any integer $i$ there exists a long exact sequence
   $$
   \xymatrix@R-1pc@C-1pc{
   0\ar[r]&\Ext^1_R(C^{i+1}(X), M)\ar[r]&\H^{-i}(\Hom_R(X,M))\ar[r]^-{\rho^i_{X,M}}&\Hom_R(\H^i(X),M)\\
   \ar[r]&\Ext^2_R(C^{i+1}(X),M)\ar[r]&\Ext^1_R(C^i(X),M)\ar[r]&\Ext^1_R(\H^i(X),M)\\
   \ar[r]&\Ext^3_R(C^{i+1}(X),M)\ar[r]&\Ext^2_R(C^i(X),M)\ar[r]&\Ext^2_R(\H^i(X),M)\ar[r]&\cdots.
   }
   $$
   \item 
   Let $X$ be a complex of flat $R$-modules.
   For any integer $i$ there exists a long exact sequence
   $$
   \xymatrix@R-1pc@C-1pc{
   \cdots\ar[r]&\Tor_2^R(\H^i(X),N)\ar[r]&\Tor_2^R(C^i(X),N)\ar[r]&\Tor_3^R(C^{i+1}(X),N)\\
   \ar[r]&\Tor_1^R(\H^i(X),N)\ar[r]&\Tor_1^R(C^i(X),N)\ar[r]&\Tor_2^R(C^{i+1}(X),N)\\
   \ar[r]&\H^i(X)\otimes_R N\ar[r]^-{\tau^i_{X,N}}&\H^i(X\otimes_R N)\ar[r]&\Tor_1^R(C^{i+1}(X),N)\ar[r]&0.
   }
   $$
   \item 
   Let $Y$ be a complex of injective $R$-modules.
   For any integer $i$ there exists a long exact sequence
   $$
   \xymatrix@R-1pc@C-1pc{
   0\ar[r]&\Ext^1_R(M,Z^{i-1}(Y))\ar[r]&\H^{i}(\Hom_R(M,Y))\ar[r]^-{\delta^i_{M,Y}}&\Hom_R(M,\H^{i}(Y))\\
   \ar[r]&\Ext^2_R(M,Z^{i-1}(Y))\ar[r]&\Ext^1_R(M,Z^{i}(Y))\ar[r]&\Ext^1_R(M,\H^{i}(Y))\\
   \ar[r]&\Ext^3_R(M,Z^{i-1}(Y))\ar[r]&\Ext^2_R(M,Z^{i}(Y))\ar[r]&\Ext^2_R(M,\H^{i}(Y))\ar[r]&\cdots.
   }
   $$
   \item
   Suppose that $R$ is commutative.
   Let $\mathfrak{a}$ be an ideal of $R$ and $Y$ a complex of injective $R$-modules.
   For any integer $i$ there exists a long exact sequence
   $$
   \xymatrix@R-1pc@C-1pc{
   0\ar[r]&\HL_{\mathfrak{a}}^1(Z^{i-1}(Y))\ar[r]&\H^{i}(\Gamma_\mathfrak{a}(Y))\ar[r]^-{\delta^i_{\mathfrak{a},Y}}&\Gamma_{\mathfrak{a}}(\H^{i}(Y))\\
   \ar[r]&\HL_{\mathfrak{a}}^2(Z^{i-1}(Y))\ar[r]&\HL_{\mathfrak{a}}^1(Z^{i}(Y))\ar[r]&\HL_{\mathfrak{a}}^1(\H^i(Y))\\
   \ar[r]&\HL_{\mathfrak{a}}^3(Z^{i-1}(Y))\ar[r]&\HL_{\mathfrak{a}}^1(Z^{i}(Y))\ar[r]&\HL_{\mathfrak{a}}^2(\H^i(Y))\ar[r]&\cdots.
   }
   $$
\end{enumerate}
\end{cor}

\begin{proof}
Applying Theorem \ref{ABtype}(1)(2) and Remark \ref{ABtypeZ} to the functors $F(-)=\Hom_R(-, M)$, $G(-)=(-)\otimes_R M$, $F_1'=\Hom_R(M,-)$ and $F_2'=\Gamma_\mathfrak{a}(-)$, respectively, we obtain the desired exact sequences.
\end{proof}

We close the section by explaining some examples of Theorem \ref{ABtype} and Corollary \ref{fundex}.
First, we recover \cite[Theorem 2.8]{AB} and \cite[Proposition 2.6(a)]{AB} from Corollary \ref{fundex}(1)(2).

\begin{ex}\label{ABex}
Let $L$ be a finitely generated $R$-module and let $\cdots P_1\to P_0\to L\to 0$ be a projective resolution of $L$ with $P_i \in \proj R$ for all $i\ge 0$.
We set 
$$
X=(0\to P_0^\ast \to P_1^\ast \to \cdots\to P_i^\ast \to\cdots).
$$
Then $X$ is a complex of projective $R$-modules.
For an $R$-module $M$, we compute the long exact sequences in Corollary \ref{fundex}(1)(2).
Now the following hold.
\begin{enumerate}[\rm(1)]
   \item 
   $C^0(X)=\Cok(0\to P_0^\ast)=P_0^\ast$ and $C^{i+1}(X)=\Cok(P_i^\ast \to P_{i+1}^\ast)=\tr\syz^i L$ for all $i\ge0$.
   \item 
   $\H^i(X)=\Ext^i(L,R)$ for all integers $i$.
   \item 
   There exist isomorphisms $\Hom(X,M)\cong X^\ast \otimes M$ and $\Hom(X^\ast, M)\cong X\otimes M$ of complexes.
   In particular, $\H^{-i}(\Hom(X,M))\cong\Tor_i(L,M)$ and $\H^i(X\otimes M)\cong\Ext^i(L,M)$ for all integers $i$.
\end{enumerate}
Therefore, for $i\ge1$, the long exact sequence
$$
\xymatrix@R-1pc@C-1pc{
0\ar[r]&\Ext^1_R(\tr\syz^i L, M)\ar[r]&\Tor_i^R(L,M)\ar[r]&\Hom_R(\Ext^i_R(L,R),M)\\
\ar[r]&\Ext^2_R(\tr\syz^i L,M)\ar[r]&\Ext^1_R(\tr\syz^{i-} L,M)\ar[r]&\Ext^1_R(\Ext^i_R(L,R),M)\\
\ar[r]&\Ext^3_R(\tr\syz^i L,M)\ar[r]&\Ext^2_R(\tr\syz^{i-1} L,M)\ar[r]&\Ext^2_R(\Ext^i_R(L,R),M)\ar[r]&\cdots
}
$$
is obtained from Corollary \ref{fundex}(1), and the long exact sequence
$$
\xymatrix@R-1pc@C-1pc{
\cdots\ar[r]&\Tor_2^R(\Ext^i_R(L,R),M)\ar[r]&\Tor_2^R(\tr\syz^{i-1} L,M)\ar[r]&\Tor_3^R(\tr\syz^i L,M)\\
\ar[r]&\Tor_1^R(\Ext^i_R(L,R),M)\ar[r]&\Tor_1^R(\tr\syz^{i-1} L,M)\ar[r]&\Tor_2^R(\tr\syz^i L,M)\\
\ar[r]&\Ext^i_R(L,R)\otimes_R M\ar[r]&\Ext^i_R(L,M)\ar[r]&\Tor_1^R(\tr\syz^i L,M)\ar[r]&0
}
$$
is obtained from Corollary \ref{fundex}(2).
These exact sequences are none other than the long versions of \cite[Theorem 2.8]{AB}.
When $i=0$, the exact sequences
$$
0\to \Ext^1_R(\tr L,M)\to L\otimes_R M\to \Hom_R(L^\ast, M)\to \Ext_R^2(\tr L,M)\to \Ext^1_R(P_0^\ast,M)=0
$$
and
$$
0=\Tor_1^R(P_0^\ast, M)\to \Tor_2^R(\tr L,M)\to L^\ast \otimes_R M \to \Hom_R(L,M)\to \Tor_1^R(\tr L,M)\to0
$$
are obtained from Theorem \ref{fundex}(1) and (2), respectively.
These are none other than \cite[Proposition 2.6(a)]{AB} and its dual sequence.
\end{ex}

To state our examples for local cohomology, we need to recall the concept of cosyzygies.
For an $R$-module $M$, we denote by $\mho^i M$ the $i$th {\em cosyzygy} of $M$, that is, $\mho^i M$ is defined as $\Ker(I^i_M \to I^{i+1}_M)$, where $I_M=(0\to I^0_M\to I^1_M \to \cdots \to I^i_M \to \cdots)$ is an injective resolution of $M$.
Note that taking $i$th cosyzygy induces an additive functor $\mho^i : \overline{\Mod R}\to \overline{\Mod R}$, where $\overline{\Mod R}$ is the injective stabilization of $\Mod R$ defined in the same way as the projective stabilization.

\begin{ex}\label{injABex}
Suppose that $R$ is commutative and let $\mathfrak{b}$ be an ideal of $R$.
Let $N$ be an $R$-module and $I=(0\to I^0\to I^1\to \cdots)$ be an injective resolution of $N$.
We set
$$
Y=\Gamma_{\mathfrak{b}}(I)=(0\to \Gamma_{\mathfrak{b}}(I^0) \to \Gamma_{\mathfrak{b}}(I^1) \to \cdots\to \Gamma_{\mathfrak{b}}(I^i) \to\cdots).
$$
By \cite[Proposition 2.1.4]{BS}, $Y$ is a complex of injective $R$-modules.
For a finitely generated $R$-module $M$ and an ideal $\mathfrak{a}$ of $R$, we compute the long exact sequence in Corollary \ref{fundex}(3)(4).
Now the following hold.
\begin{enumerate}[\rm(1)]
   \item 
   $Z^{i-1}(Y)=\Ker(\Gamma_{\mathfrak{b}}(I^{i-1})\to \Gamma_{\mathfrak{b}}(I^i))\cong\Gamma_{\mathfrak{b}}(\mho^{i-1}N)$ for all $i>0$.
   \item
   $\H^i(Y)=\HL^i_{\mathfrak{b}}(N)$ for all integers $i$.
\end{enumerate}
Also, to compute the $i$th cohomology of the complex $\Hom(M,Y)$, we recall the notion of {\em generalized local cohomology modules} in the sense of \cite{Her}.
For $R$-modules $A$, $B$ and an integer $i$, the $i$th generalized local cohomology module $\HL^i_{\mathfrak{b}}(A, B)$ with respect to $\mathfrak{b}$ is defined by 
$$
\HL^i_{\mathfrak{b}}(A, B)=\underset{n}{\varinjlim}\Ext^i(A/{\mathfrak{b}^n A}, B)\cong \H^i(\underset{n}{\varinjlim}\Hom(A/{\mathfrak{b}^n A}, I_B)), 
$$
where $I_B$ is an injective resolution of $B$.
Then the cohomology $\H^i(\Hom(M,Y))$ coincides with $\HL^i_{\mathfrak{b}}(M,N)$.
In fact, there are isomorphisms
\begin{equation*}
\begin{aligned}
\H^i(\Hom(M,Y))&=\H^i(\Hom(M,\Gamma_{\mathfrak{b}}(I)))=\H^i(\Hom(M,\underset{n}{\varinjlim}\Hom(R/{\mathfrak{b}^n}, I)))\\
&\cong \H^i(\underset{n}{\varinjlim}\Hom(M,\Hom(R/{\mathfrak{b}}^n,I)))\cong \H^i(\underset{n}{\varinjlim}\Hom(M/{\mathfrak{b}^n M}, I))\cong \HL^i_{\mathfrak{b}}(M,N),
\end{aligned}
\end{equation*}
where the third isomorphism follows from the assumption that $M$ is finitely generated.
Therefore, for $i\ge1$, the long exact sequence
$$
\xymatrix@R-1pc@C-1pc{
   0\ar[r]&\Ext^1_R(M,\Gamma_{\mathfrak{b}}(\mho^{i-1} N))\ar[r]&\HL^i_{\mathfrak{b}}(M,N)\ar[r]&\Hom_R(M,\HL^{i}_{\mathfrak{b}}(N))\\
   \ar[r]&\Ext^2_R(M,\Gamma_{\mathfrak{b}}(\mho^{i-1} N))\ar[r]&\Ext^1_R(M,\Gamma_{\mathfrak{b}}(\mho^{i} N)))\ar[r]&\Ext^1_R(M,\HL^{i}_{\mathfrak{b}}(N))\\
   \ar[r]&\Ext^3_R(M,\Gamma_{\mathfrak{b}}(\mho^{i-1} N))\ar[r]&\Ext^2_R(M,\Gamma_{\mathfrak{b}}(\mho^{i} N)))\ar[r]&\Ext^2_R(M,\HL^{i}_{\mathfrak{b}}(N))\ar[r]&\cdots
   }
$$
is obtained from Corollary \ref{fundex}(3).

Moreover, we have $\H^i(\Gamma_{\mathfrak{a}}(Y))=\H^i(\Gamma_{\mathfrak{a}}(\Gamma_{\mathfrak{b}}(I)))\cong \H^i(\Gamma_{\mathfrak{a}+\mathfrak{b}}(I))\cong \HL^i_{\mathfrak{a}+\mathfrak{b}}(N)$ by \cite[Exercise 1.1.2]{BS}.
So the long exact sequence
$$
\xymatrix@R-1pc@C-1pc{
   0\ar[r]&\HL_{\mathfrak{a}}^1(\Gamma_{\mathfrak{b}}(\mho^{i-1}N))\ar[r]&\HL^{i}_{\mathfrak{a}+\mathfrak{b}}(N)\ar[r]&\Gamma_{\mathfrak{a}}(\HL^{i}_{\mathfrak{b}}(N))\\
   \ar[r]&\HL_{\mathfrak{a}}^2(\Gamma_{\mathfrak{b}}(\mho^{i-1}N))\ar[r]&\HL_{\mathfrak{a}}^1(\Gamma_{\mathfrak{b}}(\mho^{i}N))\ar[r]&\HL_{\mathfrak{a}}^1(\HL^i_{\mathfrak{b}}(N))\\
   \ar[r]&\HL_{\mathfrak{a}}^3(\Gamma_{\mathfrak{b}}(\mho^{i-1}N))\ar[r]&\HL_{\mathfrak{a}}^2(\Gamma_{\mathfrak{b}}(\mho^{i}N))\ar[r]&\HL_{\mathfrak{a}}^2(\HL^i_{\mathfrak{b}}(N))\ar[r]&\cdots
   }
$$
is obtained from Corollary \ref{fundex}(4).
\end{ex}
%%%%%%%%%%%%%%%%%%%%%%%%%%%%%%%%%%%%%%%%%%%%%%%%%%%%%%%%%%%%
\section{Split complexes and torsionlessness and reflexivity for complexes}

Throughout the remainder of this paper, let $R$ be a two-sided noetherian ring.
We consider the homotopy category $\K(\proj R)$ consisting of complexes over $\proj R$.
%We denote by $\K(R)=K(\proj (R))$ the homotopy category consisting of complexes over $\proj(R)$.
For simplicity, the category $\K(\proj R)$ is simply denoted as $\K(R)$.

In this section, we use some results obtained in \cite{Yos}.
Here is a comment about the paper \cite{Yos}.

\begin{rmk}\label{Yosinn}
The results \cite[Theorem 1.1 and Corollary 1.3]{Yos} are incorrect.
In fact, counterexamples to these results have been obtained in \cite[Corollary 1.5]{KLOT}.
In the proofs of \cite[Theorem 1.1 and Corollary 1.3]{Yos}, \cite[Theorem 8.5]{Yos} plays an important role, and the author wonders if there are gaps in the proof of \cite[Theorem 8.5]{Yos}, as it is noted in \cite[Remark 4.3]{KLOT}.
All results obtained in this paper are independent of \cite[Theorems 1.1, 8.5 and Corollary 1.3]{Yos}.
In this paper, we use only the correct results of \cite{Yos}.
\end{rmk}

The category $\K(R)$ has the duality $(-)^\ast:\K(R)\to \K(R^{\mathrm{op}})$, i.e., for any $X=(\cdots\xrightarrow{d_X^{i-2}}X^{i-1}\xrightarrow{d_X^{i-1}}X^i\xrightarrow{d_X^i}X^{i+1}\xrightarrow{d_X^{i+1}}\cdots)\in\K(R)$ the $R$-dual complex $X^\ast=(\cdots\xrightarrow{{d_X^{i+1}}^\ast}{X^{i+1}}^\ast \xrightarrow{{d_X^i}^\ast}{X^i}^\ast \xrightarrow{{d_X^{i-1}}^\ast}{X^{i-1}}^\ast \xrightarrow{{d_X^{i-2}}^\ast}\cdots)$ is also an object of $\K(R^{\mathrm{op}})$, and $(-)^{\ast\ast}\cong\Id$ holds.
Note that for any triangle $X\to Y\to Z\to X[1]$ in $\K(R)$ the $R$-dual sequence $Z^\ast \to Y^\ast \to X^\ast \to Z^\ast[1]$ is also a triangle in $\K(R^{\mathrm{op}})$. 
We state the following remark for the $R$-dual $(-)^\ast$.
\begin{rmk}\label{boundry}
Let $X\in \K(R)$ and $i$ be an integer.
\begin{enumerate}[\rm(1)]
   \item 
   %Let $i$ be an integer.
   As there exists an exact sequence $0\to B^i(X)\to X^i\to C^i(X)\to 0$ and $X^i$ is projective, one has $B^i(X)\cong \syz C^i(X)$.
   \item 
   Dualizing the exact sequence $X^{i-1}\to X^i \to C^i(X)\to 0$ by $R$, we have the exact sequence ${X^i}^\ast \to {X^{i-1}}^\ast \to \tr C^i(X)\to 0$.
   This means that $C^{-i+1}(X^\ast)=\tr C^i(X)$.
\end{enumerate}
\end{rmk}

The cohomologically surjectivity of morphisms in $\K(R)$ is necessary to state the stable complex theory due to Yoshino \cite{Yos}.
It is also an important concept in this paper.

\begin{dfn}\cite[Definition 6.5]{Yos}\label{cohsurj}
A morphism $f:X\to Y$ in $\K(R)$ is said to be {\em cohomologically surjective} (resp. {\em cohomologically injective}) if $\H^i(f):\H^i(X)\to \H^i(Y)$ is surjective (resp. injective) for all integers $i\in\ZZ$.
\end{dfn}

For a triangle $X\xrightarrow{f} Y\xrightarrow{g} Z\xrightarrow{h} X[1]$ in $\K(R)$ we obtain the long exact sequence of cohomologies
$$
\cdots\xrightarrow{\H^{i-1}(g)}\H^{i-1}(Z)\xrightarrow{\H^{i-1}(h)}\H^i(X)\xrightarrow{\H^i(f)}\H^i(Y)\xrightarrow{\H^i(g)}\H^i(Z)\xrightarrow{\H^i(h)}\H^{i+1}(X)\xrightarrow{\H^{i+1}(f)}\cdots.
$$
The triangles $X\xrightarrow{f} Y\xrightarrow{g} Z\xrightarrow{h} X[1]$ satisfying that the sequence
$$
0\to \H^i(X)\xrightarrow{\H^i(f)} \H^i(Y)\xrightarrow{\H^i(g)}\H^i(Z)\to0
$$
is exact for any integer $i$ play a crucial role in this paper.
We often use the next terminology, which essentially states the same thing as in Definition \ref{cohsurj}.
%The triangles A and A play a crucial role in this paper, as they satisfy f = 0 for any i, i.e., A = B for any i. 
\begin{dfn}
A triangle $X\xrightarrow{f} Y\xrightarrow{g}Z\to X[1]$ in $\K(R)$ is called {\em cohomologically ghost} if the sequence $0\to \H^i(X)\xrightarrow{\H^{i}(f)}\H^i(Y)\xrightarrow{\H^i(g)}\H^i(z)\to 0$ is exact for all integers $i\in\ZZ$.
\end{dfn}

Note that the triangle $X\xrightarrow{f} Y\xrightarrow{g}Z\to X[1]$ in $\K(R)$ is cohomologically ghost if and only if the morphism $f$ is cohomologically injective, if and only if the morphism $g$ is cohomologically surjective.

The main subject of the Auslander--Bridger theory is the ideal quotient of $\mod R$ by $\proj R$.
As an analogue, Yoshino \cite{Yos} established a theory of the ideal quotient of $\K(R)$ by $\Add R$, which is defined as follows.
Objects of $\Add R$ behave like projective objects for cohomologically ghost triangles; see Lemma \ref{speq}.
Also, syzygies and cosyzygies of objects of $\K(R)$ are defined by using resolutions and coresolutions by objects of $\Add R$; see \ref{defsyzcosyz}.

\begin{dfn}\cite[Definition 5.2]{Yos}\label{AddRdef}
The full subcategory $\Add R$ of $\K(R)$ is defined as the intersection of full subcategories $\X$ of $\K(R)$ satisfying the following conditions.
\begin{enumerate}[\rm(1)]
   %\item 
   %$\X$ is closed under isomorphism in $\K(R)$, i.e., for objects $X, Y\in\K(R)$ with $X\cong Y$ in $\K(R)$, if $X\in \K(R)$, then so is $Y$.
   \item 
   The complex $R=(\cdots\xrightarrow{0}0\xrightarrow{0}R\xrightarrow{0}0\xrightarrow{0}\cdots)\in\K(R)$ is an object of $\X$.
   \item 
   $\X$ is closed under shifts, i.e., if $X\in\K(R)$, then $X[i]\in\X$ for all integers $i$.
   \item 
   $\X$ is closed under direct summands, i.e., for objects $X, Y\in\K(R)$ with $Y\in\X$, if $X$ is a direct summand of $Y$ in $\K(R)$, then $X\in\K(R)$.
   \item 
   $\X$ is closed under existing coproducts, i.e., for a family $\{X_j\}_{j\in J}$ of objects of $\X$, if the coproduct $\coprod_{j\in J}X_{j}$ exists in $\K(R)$, then $\coprod_{j\in J}X_{j}\in\X$.
\end{enumerate}
\end{dfn}

A complex $X\in\K(R)$ is called {\em split} if there exists a family $\{s^i : X^i\to X^{i-1}\}_{i\in\ZZ}$ of $R$-homomorphisms such that $d_X^{i} s^{i+1} d_X^i=d_X^i$ holds for all integers $i\in\ZZ$; see \cite[Section 1.4]{Wei} or \cite[Section 5]{Yos}.
Actually, the full subcategory of $\K(R)$ consisting of split complexes coincides with $\Add R$ defined above.
Also, split complexes can be characterized by the injectivity of the function taking cohomology modules.

\begin{thm}\cite[Lemma 5.3 and Theorem 5.8]{Yos}\label{splitdef}
Let $X\in\K(R)$.
Then the following are equivalent.
\begin{enumerate}[\rm(1)]
   \item 
   The complex $X$ belongs to $\Add R$.
   \item 
   The complex $X$ is split.
   \item
   There exists a complex $X^{\prime}\in \K(R)$ with $d_{X^\prime}=0$ such that $X\cong X^\prime$ in $\K(R)$.
   \item
   The module $C^i(X)$ is projective for all integers $i\in\ZZ$.
   \item
   The natural map $\Hom_{\K(R)}(X,Y)\to\prod_{i\in\ZZ}\Hom_R(\H^i(X), \H^i(Y))$ given by $f\mapsto (\H^i(f))_{i\in\ZZ}$ is bijective for all complexes $Y\in \K(R)$.
   \item
   The natural map $\Hom_{\K(R)}(X,Y)\to\prod_{i\in\ZZ}\Hom_R(\H^i(X), \H^i(Y))$ is injective for all complexes $Y\in \K(R)$.
\end{enumerate}
%We denote by $\Add R$ the full subcategory category of $\K(R)$ consisting of split complexes. 
\end{thm}

In \cite[Section 5]{Yos}, the ring $R$ is assumed to be commutative.
However, the comtativity of $R$ is not used in the proof of the above theorem, except the proof of the implication $(1)\Rightarrow(2)$, i.e., \cite[Theorem 5.8$(1)\Rightarrow(2)$]{Yos}.
The proof of this implication requires \cite[Proposition 5.7]{Yos}, and the proof of \cite[Proposition 5.7]{Yos} uses the commutativity of the ring $R$.
However, as we will see below, \cite[Proposition 5.7]{Yos} can be proved without using the commutativity of the ring, and thus the above theorem is true for general (not necessarily commutative) two-sided noetherian rings.

\begin{prop}\cite[Proposition 5.7]{Yos}\label{infsum}
Let $\{ X_j\}_{j\in J}$ be a family of objects of $\K(R)$ such that the differential $d_{X_j}$ of the complex $X_j$ is zero map for all $j\in J$.
Assume that the coproduct $\coprod_{j\in J}X_{j}$ in $\K(R)$ exists.
Then, for any integer $i$, the $i$th component ${X_j}^i$ is zero except for finite elements of the set $J$.
\end{prop}

\begin{proof}
We denote by $\{\alpha_{j^\prime}:X_{j^\prime}\to \coprod_{j\in J}X_j\}_{j^\prime \in J}$ the canonical morphisms.
First, we note that for any finite subset $J^\prime =\{j_1, j_2, \ldots, j_n\}$ of $J$ the morphism $(\alpha_{j_1}, \alpha_{j_2}, \ldots, \alpha_{j_n}):\bigoplus_{k=1}^n X_{j_k}\to \coprod_{j\in J}X_j$ in $\K(R)$ is split injective.
Indeed, by the universal property of coproducts, for any $j^\prime \in J$ there exists a morphism $\beta_{j^\prime}:\coprod_{j\in J}X_j \to X_{j^\prime}$ in $\K(R)$ such that $\beta_{j^\prime}\alpha_{j^\prime}=1_{X_{j^\prime}}$ and $\beta_{j^\prime}\alpha_{j^{\prime\prime}}=0$ for all $j^{\prime\prime}\in J\setminus \{j^\prime\}$.
The morphism ${}^\t\! (\beta_{j_1}, \beta_{j_2}, \ldots, \beta_{j_n}):\coprod_{j\in J}X_j\to \bigoplus_{k=1}^n X_{j_k}$ satisfies ${}^\t\! (\beta_{j_1}, \beta_{j_2}, \ldots, \beta_{j_n})(\alpha_{j_1}, \alpha_{j_2}, \ldots, \alpha_{j_n})=1_{X_{j_1}\oplus X_{j_2}\oplus\cdots\oplus X_{j_n}}$ in $\K(R)$, where ${}^\t\!(-)$ denotes the transpose of a matrix, and so the morphism $(\alpha_{j_1}, \alpha_{j_2}, \ldots, \alpha_{j_n})$ is split injective in $\K(R)$.
Also, since all differentials of $X_{j^\prime}$ are zero map for all $j^\prime \in J$, the morphism $(\alpha_{j_1}, \alpha_{j_2}, \ldots, \alpha_{j_n})$ is also split injective in $\C(R)$ by \cite[Lemma 5.6]{Yos}.
In particular, for any integer $i$ the $i$th component $(\alpha_{j_1}^i, \alpha_{j_2}^i, \ldots, \alpha_{j_n}^i)$ is injective.

Fix an integer $i$.
Assume that there exists an infinite subset $J^\prime$ of $J$ such that $X_{j^\prime}^i\ne 0$ for all $j^\prime \in J^\prime$.
Then we can take a sequence $\{j^\prime_k\}_{k=1}^{\infty}$ of elements of $J$ such that $X_{j_{k}}^i\ne0$ for all $k\ge1$.
By the injectivity of the homomorphism $(\alpha_{j_1}^i, \alpha_{j_2}^i, \ldots, \alpha_{j_k}^i) : \bigoplus_{l=1}^k X_{j_l}^i\to (\coprod_{j\in J}X_j)^i$, we can see that $\image(\alpha_{j_1}^i, \alpha_{j_2}^i, \ldots, \alpha_{j_k}^i)=\image\alpha_{j_1}^i \oplus \image\alpha_{j_2}^i\oplus\cdots\oplus\image\alpha_{j_k}^i$ and there exists an ascending chain
$$
0\subsetneq \image\alpha_{j_1}\subsetneq \image\alpha_{j_1}^i\oplus\image\alpha_{j_2}^i\subsetneq\cdots\subsetneq\image\alpha_{j_1}^i\oplus\image\alpha_{j_2}^i\oplus\cdots\oplus\image\alpha_{j_k}^i\subsetneq\cdots
$$
of submodules of $(\coprod_{j\in J}X_j)^i$.
This contradicts the fact that $(\coprod_{j\in J}X_j)^i$ is a finitely generated module over a noetherian ring.
Therefore, the module ${X_j}^i$ is zero except for finite elements of the set $J$.
\end{proof}

The next lemma follows from Theorem \ref{splitdef}, and it asserts that split complexes behave like projective objects for cohomologically ghost triangles.

\begin{lem}\label{speq}
Let $P\in\K(R)$. Then the following are equivalent.
\begin{enumerate}[\rm(1)]
   \item 
   $P\in\Add R$.
   \item 
   For any morphism $s:P\to Y$ in $\K(R)$ and any cohomologically surjective morphism $f:X\to Y$ in $\K(R)$ there exists an morphism $t:P\to X$ in $\K(R)$ such that $s=ft$ in $\K(R)$.
   \item 
   Any cohomologically ghost triangle $Z\to Z^\prime\to P\xrightarrow{s}Z[1]$ in $\K(R)$ splits, i.e., $s=0$ in $\K(R)$.
\end{enumerate}
\end{lem}

\begin{proof}
The implication $(2) \Rightarrow (3)$ is trivial.

$(1)\Rightarrow(2)$ : Let $s:P\to Y$ be a morphism in $\K(R)$ and $f:X\to Y$ a cohomologically surjective morphism in $\K(R)$. 
Let $i$ be an integer.
Since $\H^i(f):\H^i(X)\to \H^i(Y)$ is surjective and $\H^i(P)$ is projective, %by Theorem \ref{splitdef}, 
there exists a homomorphism $\widetilde{t}^i:\H^i(P)\to \H^i(X)$ of $R$-modules such that the diagram
$$
\xymatrix@R-1pc@C-1pc{
&&\H^i(P)\ar[lldd]_-{\widetilde{t}^i}\ar[dd]^-{\H^i(s)}\\
&&\\
\H^i(X)\ar[rr]^-{\H^i(f)}&&\H^i(Y)
}
$$
commutes.
Since $P\in\Add R$, there exists a morphism $t:P\to X$ in $\K(R)$ such that $\widetilde{t}^i=\H^i(t)$ for all integers $i$ by Theorem \ref{splitdef}.
As the equality $\H^i(ft)=\H^i(s)$ holds for all integers $i$, by Theorem \ref{splitdef} again, we have $ft=s$.

$(3)\Rightarrow(1)$ : By Theorem \ref{splitdef}, it is enough to prove that the natural morphism $H : \Hom_{\K(R)}(X,Y)\to\prod_{i\in\ZZ}\Hom_R(\H^i(X), \H^i(Y))$ is injective for all $Y\in\K(R)$.
Let $s:P\to Y$ be a morphism in $\K(R)$ such that $H(s)=0$.
By the assumption, the triangle $Y[-1]\to (\Cone s)[-1]\to P\xrightarrow{s}Y$ in $\K(R)$ splits, i.e., $s=0$ in $\K(R)$.
We are done.
\end{proof}

The (projective) stabilization $\underline{\mod R}$ of $\mod R$ is defined by the ideal quotient of $\mod R$ by the category $\proj R$ of finitely generated projective modules.
As an analogue, the stabilization $\underline{\K(R)}$ of $\K(R)$ is defined by the ideal quotient of $\K(R)$ by the category $\Add R$ of split complexes.

\begin{dfn}\cite[Definition 6.1]{Yos}\label{Kvar}
We denote by $\underline{\K(R)}$ the ideal quotient of $\K(R)$ by $\Add R$.
The objects of $\underline{\K(R)}$ are the same as those of $\K(R)$.
The morphism set of objects $X, Y$ of $\K(R)$ is defined by
$$
\Hom_{\underline{\K(R)}}(X,Y)=\Hom_{\K(R)}(X,Y)/{\mathcal{A}(X,Y)},
$$
where $\mathcal{A}(X,Y)$ is the set of morphisms in $\K(R)$ factoring through some split complexes.
Also, we say that objects $X$ and $Y$ of $\K(R)$ are {\em stably isomorphic} if $X$ and $Y$ are isomorhic as objects of $\underline{\K(R)}$, and then write $X\approx Y$.
\end{dfn}

For a finitely generated $R$-module $M$, the kernel of a right $\proj R$-approximation $P\to M$ (i.e., a surjective $R$-homomorphism $P\to M$ with $P\in\proj R$) is called the syzygy of $M$.
For a complex of finitely generated projective $R$-module X, the cocone of right $\Add R$-approximation is called the syzygy of $X$.
The cosyzygy of $X$ is defined similarly.

\begin{dfn}\cite[Definitions 7.1, 7.3, 7.4 and 7.8]{Yos}\label{defsyzcosyz}
Let $X\in\K(R)$.
\begin{enumerate}[\rm(1)]
   \item
   A morphism $p:P\to X$ (resp. $q:X\to Q$) in $\K(R)$ is called a {\em right $\Add R$-approximation} (resp. {\em left $\Add R$-approximation }) of $X$ if $P\in\Add R$ (resp. $Q\in\Add R$) and the map $\Hom_{\K(R)}(W,p):\Hom_{\K(R)}(W,P)\to\Hom_{\K(R)}(W, X)$ (resp. $\Hom_{\K(R)}(q,W):\Hom_{\K(R)}(Q,W)\to\Hom_{\K(R)}(X,Q)$) is surjective for all $W\in \Add R$.
   \item
   Let $p:P\to X$ (resp. $q:X\to Q$) in $\K(R)$ be a right (resp. left) $\Add R$-approximation of $X$.
The {\em syzygy} $\syz X$ (resp. {\em cosyzygy} $\syz^{-}X$) of $X$ is defined as $\Cone(p)[-1]$ (resp. $\Cone(q)$). 
Then taking the syzygy (resp. cosyzygy) induces an additive functor $\syz:\underline{\K(R)}\to\underline{\K(R)}$ (resp. $\syz^{-}:\underline{\K(R)}\to\underline{\K(R)}$).
For an integer $n>0$, we define inductively $\syz^n=\syz\circ\syz^{n-1}$ (resp. $\syz^{-n}=\syz^{-}\circ\syz^{-n+1}$).
\end{enumerate}
\end{dfn}

\begin{rmk}\label{syzcos}
Let $X\in\K(R)$.
\begin{enumerate}[\rm(1)]
   \item\cite[Lemmas 7.2 and 7.5]{Yos}
   Let $P$, $Q\in\Add R$.
Then a morphism $p:P\to X$ in $\K(R)$ is a right $\Add R$-approximation of $X$ if and only if it is cohomologically surjective.
Also, a morphism $q:X\to Q$ in $\K(R)$ is a left $\Add R$- approximation of $X$ if and only if its $R$-dual $q^\ast: Q^\ast \to X^\ast$ is cohomologically surjective.
\item
Take a surjective morphism $s^i:P^i\to \H^i(X)$ of $R$-modules with $P^i$ projective for all integers $i$. Then there exists a morphism $\tilde{p}^i:P^i\to Z^i(X)$ such that $\pi^i_{ZH}\tilde{p}^i=s^i$ for all integers $i$, where $\pi^i_{ZH}:Z^i(X)\to \H^i(X)$ is the natural surjection.
We consider the composite $p^i=\iota_{ZX}^i\tilde{p}^i:P^i\to X^i$, where $\iota^i_{ZX}:Z^i(X)\to X^i$ is the natural injection. The morphism of complexes
$$
\xymatrix@R-1pc@C-1pc{
P\ar[dd]^{p}&=&(&\cdots\ar[r]^-{0}&P^{i-1}\ar[r]^-{0}\ar[dd]^-{p^{i-1}}&P^i\ar[r]^-{0}\ar[dd]^-{p^{i}}&P^{i+1}\ar[r]^-{0}\ar[dd]^-{p^{i+1}}&\cdots&)\\
&&&&&&&&\\
X&=&(&\cdots\ar[r]&X^{i-1}\ar[r]^-{d_X^{i-1}}&X^i\ar[r]^-{d_X^i}&X^{i+1}\ar[r]&\cdots&)
}
$$
is a right $\Add R$-approximation of $X$ since it is cohomologically surjective. Therefore, the syzygy $\syz X$ of X can be computed as
$$
\syz X=(\cdots\to P^{i-1}\oplus X^{i-2}\xrightarrow{ \begin{pmatrix}0& 0 \\
p^{i-1} & d_X^{i-2} \end{pmatrix}}P^i\oplus X^{i-1}\xrightarrow{ \begin{pmatrix}0& 0 \\ p^{i} & d_X^{i-1} \end{pmatrix}}P^{i+1}\oplus X^i \to\cdots).
$$
Also, noting that $\syz^{-}X=(\syz(X^\ast))^\ast$, we can compute the cosyzygy $\syz^{-}X$ similarly.
\end{enumerate}
\end{rmk}

The syzygy for complexes defined above commutes with the cohomology $H(-)$ and the cokernel $C(-)$.

\begin{lem}\label{Csyz}
Let $X\in\K(R)$ and $i$ be an integer.
Then one has stable isomorphisms $\H^i(\syz X)\approx\syz \H^i(X)$, $C^i(\syz X)\approx\syz C^i(X)$.
\end{lem}

\begin{proof}
We use the same notation as Remark \ref{syzcos}(2). %大丈夫か
Since there exists a triangle $\syz X\to P\xrightarrow{p}X\to\syz X[1]$ with $p$ cohomologically surjective, we have the exact sequence
$$
0\to \H^{\bullet}(\syz X)\to \H^{\bullet}(P)\to \H^{\bullet}(X)\to0.
$$
As $\H^{\bullet}(P)$ is projective, one has $\H^{\bullet}(\syz X)\approx \syz \H^{\bullet}(X)$.

Let $i$ be an integer.
As it is seen is Remark \ref{syzcos}(2), the module $C^i(X)$ is isomorphic to the cokernel of the morphism $\begin{pmatrix}0& 0 \\ p^{i-1} & d_X^{i-2}\end{pmatrix}:P^{i-1}\oplus X^{i-2}\to P^i\oplus X^{i-1}$.
It is easy to see that the image of $(p^{i-1}, d_X^{i-2}):F^{i-1}\oplus X^{i-2}\to X^{i-1}$ coincides with $Z^{i-1}(X)$.
We have
$$
C^i(\syz X)\cong \Cok\left(\begin{pmatrix}0& 0 \\ p^{i-1} & d_X^{i-2}\end{pmatrix}\right)\cong P^i\oplus (X^{i-1}/Z^{i-1}(X))\cong P^{i}\oplus B^i(X)\approx \syz C^i(X).
$$
\end{proof}

Let $X\in\K(R)$ and $i$ be an integer.
We denote by $\gamma^i_X$ the natural map $\H^{-i}(X)\to \H^{i}(X^\ast)^\ast$ given by $\gamma_X^i(\overline{x})(\overline{f})=f(x)$ for $\overline{x}\in \H^{-i}(X)$ and $\overline{f}\in \H^{i}(X^\ast)$.
Note that the homomorphism $\gamma_{X}^i$ coincides with $\rho_{X^{\ast}, R}^i$ in Corollary \ref{fundex}.

In \cite{Yos}, a complex $X\in\K(R)$ is called {\em ${}^\ast$torsionfree} if the natural map $\rho_X^i:=\rho_{X, R}^i : \H^{-i}(X^\ast)\to \H^i(X)^\ast$ is injective for all integers $i$ and called {\em ${}^\ast$reflexive} if $\rho_{X, R}^i$ is bijective for all $i$.
Here, we introduce torsionlessness and reflexivity for complexes by the injectivity and bijectivity of the natural map $\gamma^{\bullet}_X$ as the dual notions of the ${}^\ast$torsionfreeness and ${}^\ast$reflexive.
These notions are natural analogies of the concepts of torsionless modules and reflexive modules; see Propositions \ref{trleq} and \ref{refeq}.
%By the injectivity and bijectivity of the natural map $\gamma^{\bullet}_X$, we define the notions of torsionless complexes and reflexive complexes.

\begin{dfn}
Let $X\in\K(R)$.
\begin{enumerate}
   \item
   The complex $X$ is {\em torsionless} if the natural map $\gamma^i_X: \H^{-i}(X)\to \H^i(X^\ast)^\ast$ is injective for all integers $i$.
   \item
   The complex $X$ is {\em torsionless} if the natural map $\gamma^i_X: \H^{-i}(X)\to \H^i(X^\ast)^\ast$ is injective for all integers $i$.
\end{enumerate}
\end{dfn}

Let $M$ be a finitely generated $R$-module.
It is well-known that $M$ is torsionless if and only if there exists an exact sequence $0\to M\to P\to M_1\to 0$ of finitely generated $R$-modules with $P\in\proj R$ such that the $R$-dual $0\to M^\ast_1 \to P^\ast \to M^\ast \to 0$ is also exact.
The following proposition states that the torsionlessness of complexes of finitely generated projective $R$-modules is equivalent to the existence of a similar cohomologically ghost triangle.

\begin{prop}\label{trleq}
Let $X\in \K(R)$. The following are equivalent.
\begin{enumerate}[\rm(1)]
   \item
   The complex $X$ is torsionless.
   %\item
   %The natural map $\H^{-i}(X)\to \H^i(X^\ast)^\ast$ is injective for all integers $i\in\ZZ$.
   \item 
   %There exist a complex $P\in\Add R$ and a cohomologically injective morphism $q:X\to P$ in $K(R)$ such that $q^\ast:P^\ast \to X^\ast$ is cohomologically surjective.
   There exists a cohomologically ghost triangle $X\xrightarrow{q} P\xrightarrow{r} X_1\to X[1]$ in $\K(R)$ with $P\in\Add R$ such that the $R$-dual triangle ${X_1}^\ast \xrightarrow{r^\ast} P^\ast \xrightarrow{q^\ast} X^\ast \to {X_1}^\ast[1]$ is also cohomologically ghost.
   \item
   %There exist a complex $P\in\Add R$ and a cohomologically injective morphism $X\to P$ in $K(R)$.
   There exists a cohomologically ghost triangle $X\xrightarrow{q} P\xrightarrow{r} X_1\to X[1]$ in $\K(R)$ with $P\in\Add R$.
   \item
   The complex $X$ is a first syzygy, i.e. $X\approx\syz X_1$ for some $X_1\in \K(R)$.
\end{enumerate}
\end{prop}

\begin{proof}
%The equivalence $(1)\Leftrightarrow(2)$ follows from Theorem \ref{torref}, and the equivalence $(3)\Rightarrow(4)\Leftrightarrow(5)$ is clear.
The implications $(2)\Rightarrow(3)\Leftrightarrow(4)$ clearly hold.
We assume that the condition (4) holds.
%Then the module $\H^i(X)$ is a submodule of a projective module $P^i$ for all integers $i$.
%In other words, the natural map $\H^i(X)\to \H^i(X)^{\ast\ast}$ is injective for all $i$.
%It is easily to see that the diagram %\varphi_{}
%$$
%\xymatrix@R-1pc@C-1pc{
%\H^{\bullet}(X)\ar[dd]^-{\varphi_{\H^{\bullet}(X)}}\ar@{=}[rr]&&\H^{\bullet}(X)\ar[dd]^-{\gamma_X^{\bullet}}\\
%&&\\
%\H^{\bullet}(X)^{\ast\ast}\ar[rr]^-{{\rho_X^{\bullet}}^\ast}&&\H^{-{\bullet}}(X^\ast)^\ast
%}
%$$
%is commutative.%ATODEATODE
Then there exists a commutative diagram
$$
\xymatrix@R-1pc@C-1pc{
\H^{-\bullet}(X)\ar[dd]^-{\gamma^{\bullet}_X}\ar[rr]^-{\H^{-\bullet}(q)}&&\H^{-\bullet}(P)\ar[dd]^-{\gamma^{\bullet}_P}\\
&&&\\
\H^{\bullet}(X^\ast)^\ast\ar[rr]^-{\H^{\bullet}(q^\ast)^\ast} &&\H^{\bullet}(P^\ast)^\ast .
}
$$
Since $\H^{-\bullet}(q)$ and $\gamma_P^{\bullet}$ are injective, so is $\gamma_X^{\bullet}$.
Hence $X$ is torsionless and the implication $(3)\Rightarrow(1)$ holds.
We prove the implication $(1)\Rightarrow(2)$. Take a right $\Add(R^{\mathrm{op}})$-approximation $p:Q\to X^\ast$ of $X^\ast$.
Then it is cohomologically surjective by Remark \ref{syzcos}(1), and there exists a commutative diagram
$$
\xymatrix@R-1pc@C-1pc{
\H^{-\bullet}(X)\ar[dd]^-{\gamma^{\bullet}_X}\ar[rr]^-{\H^{-\bullet}(p^\ast)}&&\H^{-\bullet}(Q^\ast)\ar[dd]^-{\gamma^{\bullet}_{Q^\ast}}\\
&&&\\
\H^{\bullet}(X^\ast)^\ast\ar[rr]^-{\H^{\bullet}(p)^\ast} &&\H^{\bullet}(Q)^\ast .
}
$$
By the assumption (2), $\gamma^{\bullet}_X$ is injective.
Also, as $p$ is cohomologically surjective, $\H^{\bullet}(p)^\ast$ is injective.
We conclude that $\H^{\bullet}(p^\ast)$ is injective, that is, $p^\ast : X\cong X^{\ast\ast}\to Q^\ast$ is cohomologically injective.
So we obtain the desired triangle $X\xrightarrow{p^\ast}Q^\ast \to \Cone(p^\ast)\to X[1]$.
%Then, for any integer $i$ there exists an injective morphism $C^i(X)\hookrightarrow P^i$ with $P^i$ projective.
%Since the diagram
%$$
%\xymatrix@R-1pc@C-1pc{
%X^i\ar@{->>}[r]\ar[d]^-{d_X^i}&C^i(X)\ar[r]\ar[d]^-{0}&P^i\ar[d]^-{0}\\
%X^{i+1}\ar@{->>}[r]&C^{i+1}(X)\ar[r]&P^{i+1}
%}
%$$
%is commutative, the induced map $f:X\to P$ is a morphism of complexes, where $P=(\cdots\xrightarrow{0}P^i\xrightarrow{0}P^{i+1}\xrightarrow{0}\cdots)$.
%Clearly, $f$ is cohomologically injective.
\end{proof}

We make a similar argument for the reflexivity of objects of $\K(R)$.

\begin{prop}\label{refeq}
Let $X\in\K(R)$.
The following are equivalent.
\begin{enumerate}[\rm(1)]
   \item
   The complex $X$ is reflexive.
%   \item
%   The module $C^i(X)$ is reflexive for all integers $i$.
   \item
   There exist two cohomologically ghost triangles $X\to P_0\to X_1\to X[1]$ and $X_1\to P_1\to X_2\to X_1[1]$ in $\K(R)$ with $P_0$, $P_1\in\Add R$ such that the $R$-dual triangles $X_1^\ast \to P_0^\ast \to X^\ast \to X_1^\ast [1]$ and $X_2^\ast \to P_1^\ast \to X_1^\ast \to X_2^\ast [1]$ are also cohomologically ghost.
%↑大丈夫か with
\end{enumerate}
\end{prop}

\begin{proof}
%The equivalence $(1)\Leftrightarrow(2)$ follows from Theorem ATODEATODE(longexnorversion).
We prove the implication $(1)\Rightarrow(2)$.
By Proposition \ref{trleq}, there exists a cohomologically ghost triangle $X\to P_0\to X_1\to X[1]$ such that the $R$-dual is also cohomologically ghost.
We get a commutative diagram
\begin{equation}\label{33}
\xymatrix@R-1pc@C-1pc{
0\ar[rr]&&\H^{-\bullet}(X)\ar[rr]\ar[dd]^-{\gamma^{\bullet}_{X}}&&\H^{-\bullet}(P_0)\ar[rr]\ar[dd]^-{\gamma^{\bullet}_{P_0}}&&\H^{-\bullet}(X_1)\ar[rr]\ar[dd]^-{\gamma^{\bullet}_{X_1}}&&0\\
&&&&&&&&\\
0\ar[rr]&&\H^{\bullet}(X^\ast)^\ast\ar[rr]&&\H^{\bullet}(P_0^\ast)^\ast\ar[rr]&&\H^{\bullet}(X_1^\ast)^\ast
}
\end{equation}
with exact rows.
Since $\gamma^{\bullet}_X$ and $\gamma^{\bullet}_{P_0}$ are bijective, $\gamma^{\bullet}_{X_1}$ is injective, i.e., $X_1$ is torsionless.
Applying Proposition \ref{trleq} to $X_1$, we obtain the desired triangles.
Next, we show the converse. Assume the condition $(2)$ holds true.
Then, there exists a diagram similar to (\ref{33}) for the triangle $X\to P_0\to X_1\to X[1]$ obtained from the condition (3).
%↑大丈夫か　holds true totoとかsimilar toとか
Then $\gamma^{\bullet}_{P_0}$ is bijective, and $\gamma^{\bullet}_{X_1}$ is injective by the assumption (2) and Proposition \ref{trleq}.
Hence $\gamma^{\bullet}_X$ is bijective, i.e., $X$ is reflexive. We are done.
\end{proof}

%%%%%%%%%%%%%%%%%%%%%%%%%%%%%%%%%%%%%%%%%%%%%%%%%%%%%%%%%%%%
\section{Higher torsionfreeness, Gorenstein projective complexes and approximation theory}

In the previous section, we give the definitions of the notions of torsionless complexes and reflexive complexes.
%In this section, as a higher version of these notions, the notion of $n$-torsionfree complexes is introduced.
In this section, we introduce the notion of $n$-torsionfree complexes as a higher version of these notions.
We say that an object $X$ of $\K(R)$ is Gorenstein projective if $X$ and $X^\ast$ are $\infty$-torsionfree.
As an analogue of the Gorenstein dimension for finitely generated modules, we define the Gorenstein dimension for complexes and prove approximation theorems for $n$-torsionfree complexes and complexes of finite Gorenstein dimension.
%前のセクションで、torsionless complexesやreflexive complexesの概念達を導入した。このセクションでは、それらの概念の高次元版として、n-torsionfree complexたちの概念を導入する。その定義に基づき、複体Xに対し、XとX*が無限torsionfreeであるときに、XがGorenstein projectiveとよぶ。加群に対するGorenstein dimensionの類似として、複体に対するGorenstein dimensionを定義し、n-trf複体たちや有限なGorenstein dimをもつ複体達に対する近似定理を証明する。
Our main tool is the following long exact sequence obtained in Section 3.
%我々の主要なツールは、次の長官前列だ、セクション3で得られ

\begin{cor}\label{starex}
Let $X\in \K(R)$ and $i$ be an integer.
There exists a long exact sequence
$$
\xymatrix@R-1pc@C-1pc{
0\ar[r]&\Ext^1_{R^{\mathrm{op}}}(\tr C^{-i}(X), R)\ar[r]&\H^{-i}(X)\ar[r]^-{\gamma^i_X}&\H^i(X^\ast)^\ast\\
\ar[r]&\Ext^2_{R^{\mathrm{op}}}(\tr C^{-i}(X), R)\ar[r]&\Ext^1_{R^{\mathrm{op}}}(\tr C^{-i+1}(X), R)\ar[r]&\Ext^1_{R^{\mathrm{op}}}(\H^i(X^\ast),R)\\
\ar[r]&\Ext^3_{R^{\mathrm{op}}}(\tr C^{-i}(X), R)\ar[r]&\Ext^2_{R^{\mathrm{op}}}(\tr C^{-i+1}(X), R)\ar[r]&\Ext^2_{R^{\mathrm{op}}}(\H^i(X^\ast),R)\ar[r]&\cdots.
}
$$
\end{cor}

The following corollary is the dual version of \cite[Theorem 3.4]{Yos} and follows from Corollary \ref{starex}.
Note that this corollary asserts that the converse of \cite[Theorem 3.4(2)]{Yos} is also true. 
%Note that the following system asserts that the converse claim of Yoshino Theorem 3.4 (2) is also true.
%次の系は、吉野定理3.4の双対版であり、直ちに系\ref{starex}から得られる。我々は吉野定理2.3の長完全を記述したおかげで、吉野定理3.4（2）の逆の主張も正しい事を得る。

\begin{cor}\label{torref}
Let $X\in K(R)$.
The following hold.
\begin{enumerate}
   \item
   The complex $X$ is torsionless if and only if $\Ext^1_R(C^i(X^\ast),R)=0$ for all integers $i\in\ZZ$, i.e., $C^i(X)$ is a torsionless module for all integers $i\in\ZZ$. 
   %the natural map $\H^{-i}(X)\to \H^i(X^\ast)^\ast$ is injective for all integers $i\in\ZZ$.
   \item
   The complex $X$ is reflexive if and only if $\Ext^1_R(C^i(X^\ast),R)=\Ext^2(C^i(X^\ast), R)=0$ for all integers $i\in\ZZ$, i.e., $C^i(X)$ is a reflexive module for all integers $i\in\ZZ$. 
   %The complex $X$ is reflexive if and only if the natural map $\H^{-i}(X)\to \H^i(X^\ast)^\ast$ is bijective for all integers $i\in\ZZ$.
\end{enumerate}
\end{cor}

Motivated by the above corollary, we introduce the $n$-torsionfreeness for objects of $\K(R)$ for any positive integer $n\ge0$.
The $1$-torsionfreeness is equivalent to the torsionlessness, and the $2$-torsionfreeness is equivalent to the reflexivity.
%系\ref{torref}に動機付けられ、複体達に対するn-torsionfree性を導入する、1-torsionfree性がtorsionless、。

\begin{dfn}\label{defntrf}
Let $n$ be a positive integer.
A complex $X\in \K(R)$ is said to be {\em $n$-torsionfree} if $\Ext^j_{R^{\mathrm{op}}}(C^i(X^\ast), R)=0$ for all integers $i\in\ZZ$ and $1\le j\le k$.
\end{dfn}

The following lemma immediately follows from Corollary \ref{starex}.

\begin{lem}\label{ntrftriv}
Let $X\in\K(R)$ and $n>1$ be an integer.
Then $X$ is $n$-torsionfree if and only if the natural map $\gamma_{X}^i : \H^{-i}(X)\to \H^i(X^\ast)^\ast$ is bijective and $\Ext_{R^{\mathrm{op}}}^j(\H^{i}(X^\ast) , R)=0$ for all $i\in \ZZ$ and $1\le j\le n-2$.
\end{lem}

Let $M$ be a finitely generated $R$-module and $n>0$ an integer.
Auslander and Bridger \cite{AB} proved that $M$ is $n$-torsionfree if and only if there exists an exact sequence $0\to M\to P_0\to P_1 \to \cdots \to P_{n-1}$ with $P_i\in\proj R$ for all $i$ such that the $R$-dual $P_{n-1}^\ast \to \cdots \to P_1^\ast \to P_0^\ast \to M^\ast \to 0$ is also exact.
The following theorem is an analogue of the result of Auslander and Bridger for complexes, and gives a higher version of Propositions \ref{trleq} and \ref{refeq}.

\begin{thm}\label{ntrfeq}
Let $X\in \K(R)$ and $n\ge 1$ be an integer.
Then the following are equivalent.
\begin{enumerate}[\rm(1)]
   \item
   The complex $X$ is $n$-torsionfree.
   \item
   For each $0\le i\le n-1$, there exists a cohomologically ghost triangle $X_i\to P_i\to X_{i+1}\to X_i[1]$ in $\K(\proj R)$ with $P_i\in\Add R$ whose $R$-dual $X_{i+1}^\ast\to P_i^\ast\to X_i^\ast\to X_{i+1}^\ast[1]$ is also cohomologically ghost, such that $X\cong X_0$ in $\K(R)$.
   %There are cohomologically ghost triangles $X_i\to P_i\to X_{i+1}\to X_i[1]$ in $\K(R)$ with $P_i\in\Add R$ for $0\le i\le n-1$ such that $X\cong X_0$ in $\K(R)$ and the $R$-dual triangle $X_{i+1}^\ast \to P_i^\ast \to X_i^\ast \to X_{i+1}^\ast[1]$ is also cohomologically ghost for all $0\le i\le n-1$.
\end{enumerate}
\end{thm}

\begin{proof}
We prove the assertion by using induction on $n$.
From Propositions \ref{trleq} and \ref{refeq}, we know that the equivalence $(1)\Leftrightarrow(2)$ holds true for $n=1, 2$.
Assume that $n\ge 3$.
First, we prove the implication $(1)\Rightarrow(2)$.
Suppose that $X$ is $n$-torsionfree and take a cohomologically ghost triangle $X\to P_0\to X_1\to X[1]$ such that its $R$-dual $X_1^\ast \to P_0^\ast \to X^\ast \to X_1^\ast [1]$ is also cohomologically ghost.
Then we have a commutative diagram
\begin{equation}\label{44}
\xymatrix@R-1pc@C-1pc{
0\ar[rr]&&\H^{-\bullet}(X)\ar[rr]\ar[dd]^-{\gamma^{\bullet}_{X}}&&\H^{-\bullet}(P_0)\ar[rr]\ar[dd]^-{\gamma^{\bullet}_{P_0}}&&\H^{-\bullet}(X_1)\ar[rr]\ar[dd]^-{\gamma^{\bullet}_{X_1}}&&0&&\\
&&&&&&&&\\
0\ar[rr]&&\H^{\bullet}(X^\ast)^\ast\ar[rr]&&\H^{\bullet}(P_0^\ast)^\ast\ar[rr]&&\H^{\bullet}(X_1^\ast)^\ast\ar[rr]&&\Ext^1(\H^{\bullet}(X^\ast), R)\ar[rr]&&0
}
\end{equation}
with exact rows.
The morphisms $\gamma^{\bullet}_X$ and $\gamma^{\bullet}_P$ are bijective and $\Ext^1(\H^{^\bullet}(X^\ast), R)=0$ by the 3-torsionfreeness of $X$ and Lemma \ref{ntrftriv}. 
Hence $\gamma^{\bullet}_{X_1}$ is bijective.
Also, since the sequence $0\to \H^{\bullet}(X_1^\ast)\to \H^{\bullet}(P_0^\ast)\to \H^{\bullet}(X^\ast)$ is exact and $\H^{\bullet}(P_0^\ast)$ is projective, we have $\Ext^{i}(\H^{\bullet}(X_1^\ast),R)\cong \Ext^{i+1}(\H^{\bullet}(X^\ast),R)=0$ for all $1\le i\le n-3$.
By virtue of Lemma \ref{ntrftriv}, $X_1$ is $(n-1)$-torsionfree.
Applying the induction hypothesis to $X_1$, we get the desired triangles.

Next, we show that the converse holds true.
For the cohomologically ghost triangle $X\to P_0\to X_1\to X[1]$ obtained from the assumption $(2)$, the complexes $X$ and $X_1$ are $(n-1)$-torsionfree by the induction hypothesis.
Thus it is enough to show that $\Ext^{n-2}(\H^{\bullet}(X^\ast),R)=0$.
When $n=3$, since $\gamma^{\bullet}_{X_1}$ is bijective by the reflexivity of $X_1$, it follows from the diagram (\ref{44}) that $\Ext^1(\H^{\bullet}(X^\ast),R)=0$.
When $n>3$, we have $\Ext^{n-2}(\H^{\bullet}(X^\ast),R)\cong\Ext^{n-3}( \H^{\bullet}(X_1^\ast),R)$.
But these extensions vanish as $X_1$ is $(n-1)$-torsionfree.
Therefore, $X$ is $n$-torsionfree and we are done.
\end{proof}

A finitely generated $R$-module $M$ is called Gorenstein projective (or totally reflexive) is $M$ and $\tr M$ are $\infty$-torsionfree.
We define the Goresntein projectivity for objects of $\K(R)$.
Moreover, using the concept of cohomologically ghost triangles, we define the projective (resp. Gorenstein projective) dimension for objects of $\K(R)$ as an analogue of projective (resp. Gorenstein projective) dimension for finitely generated $R$-modules.
%A finitely generated $R$-module $M$ is called Gorenstein projective (or totally reflexive) is $M$ and $\tr M$ are $\infty$-torsionfree.
%複体に対するn-torsionfree性の定義(Definition \ref{defntrf})に基づき、我々はK(R)の対象に対し、the Goresntein projectivity for objects of $\K(R)$を定義する。その上、using cohomologically ghost triangles, we define the projective (resp. Gorenstein projective) dimension for objects of $\K(R)$ as an analogue of projective (resp. Gorenstein projecive) dimension for finitely generated $R$-modules.

%\section{Projective dimension of complexes}
\begin{dfn}\label{defGdim}
Let $X\in \K(R)$.
\begin{enumerate}[\rm(1)]
   \item
   The {\em projective dimension} $\pd_R X$ of the complex $X$ is defined to be the infimum of integers $n$ such that there exist triangles $X_{i+1}\to P_i\xrightarrow{f_i} X_i\to X_{i+1}[1]$ with $P_i$ split and $f_i$ cohomologically surjective for all $0\le i\le n-1$ such that $X_0\cong X$ in $\K(R)$ and $X_n$ is split.
   \item
   The complex $X$ is called {\em Gorenstein projective} if $C^i(X)$ is Gorenstein projective for all integers $i$, i.e., $\Ext_R^j(C^i(X),R)=0$ and $\Ext^j_R(C^i(X^\ast),R)=0$ for all integers $i\in\ZZ$ and $j>0$.
   \item
   The {\em Gorenstein dimension} $\Gdim_R X$ of the complex $X$ is defined to be the infimum of integers $n$ such that there exist triangles $X_{i+1}\to G_i\xrightarrow{f_i} X_i\to X_{i+1}[1]$ with $G_i$ Gorenstein projective and $f_i$ cohomologically surjective for all $0\le i\le n-1$ such that $X_0\cong X$ in $\K(R)$ and $X_n$ is Gorenstein projective.
\end{enumerate}
\end{dfn}

For any $X\in\K(R)$ and integer $n\ge0$, $\pd_R X\le n$ if and only if $\syz^n X\in\Add R$ by Definition \ref{defsyzcosyz} and Remark \ref{syzcos}(1).

We need the following lemma, which is well-known for short exact sequences, to state several basic properties of projective dimension and Gorenstein dimension.
\begin{lem}\label{mawasu}
Let $X\xrightarrow{f}Y\xrightarrow{g}Z\to X[1]$ be a cohomologically ghost triangle in $\K(R)$.
Then there exists a cohomologically ghost triangle $\syz Z\to X\oplus P\xrightarrow{(f, s)}Y\to \syz Z[1]$ in $\K(R)$ with $P\in\Add R$.
\end{lem}

\begin{proof}
We take a split complex $P\in \K(R)$ and a morphism $s:P\to Z$ in $\K(R)$ so that $s$ is cohomologically surjective.
By Lemma \ref{speq}, there exists a morphism $s:P\to Y$ in $\K(R)$ such that $fs=t$ in $\K(R)$.
Then, the diagram
$$
\xymatrix@R-1pc@C-1pc{
X\ar@{=}[dd]\ar[rr]^{\binom{1}{0}}&&X\oplus P\ar[dd]^-{(f, s)}\ar[rr]^-{(0, 1)}&&P\ar[dd]^-{t}\\
&&&\\
X\ar[rr]^{f}&&Y\ar[rr]^{g}&&Z
}
$$
is commutative, the morphism $(f, s)$ is also cohomologically surjective.
By the octahedron axiom, there are two morphisms $a:P\to Z$ and $b:Z\to \Cone(f,s)$ such that the diagram
$$
\xymatrix@R-1pc@C-1pc{
X\ar@{=}[dd]\ar[rr]^{\binom{1}{0}}&&X\oplus P\ar[dd]^-{(f, s)}\ar[rr]^-{(0, 1)}&& P \ar[rr]^-{0}\ar[dd]^-{a}&&X[1]\ar@{=}[dd]\\
&&&\\
X\ar[rr]^{f}&&Y\ar[dd]\ar[rr]^{g}&&Z\ar[dd]^-{b}\ar[rr]&&X[1]\\
&&&\\
&&C_{(f,s)}\ar[dd]\ar@{=}[rr]&&C_{(f,s)}\ar[dd]&&\\
&&&\\
&&(X\oplus P)[1]\ar[rr]&&P[1],&&
}
$$
where $C_{(f,s)}=\Cone(f,s)$.
As the morphisms $g$ and $(f, s)$ are cohomologically surjective, so is the morphism $a$.
Namely, $C_{(f,s)}\cong \syz Z$ in $\K(R)$ up to projective summands.
Consequently, we obtain the triangle $\syz Z\to X\oplus P\xrightarrow{(f,s)}Y\to \syz Z[1]$ and $(f,s)$ is cohomologically surjective.
\end{proof}

A full subcategory $\X$ of $\mod R$ is called {\em resolving} if it contains the module $R$ and is closed under direct summands, extensions and kernels of epimorphisms.
It was proved by Auslander and Bridger that the full subcategory $\gp R$ of $\mod R$ consisting of Gorenstein projective $R$-modules is resolving.
The third assertion of the following proposition states the corresponding result for the category $\K(R)$.

\begin{prop}\label{resol}
   The following hold.
   \begin{enumerate}[\rm(1)]
   \item
   Let $X\in\K(R)$.
   Then the equality $\pd_RX=\sup\{\pd_R C^i(X) \mid i\in\ZZ \}$ holds.
   \item
   Let $X\in\K(R)$.
   If $X$ is Gorenstein projective, then $\H^i(X)$ is a Gorenstein projective module for all integers $i$.
   \item
   Let $X$, $Y$ and $Z$ be objects of $\K(R)$.
   Then the following hold.
   \begin{itemize}
   \item[(3-a)]
   Split complexes are Gorenstein projective.
   \item[(3-b)]
   Let $Y\in\K(R)$.
   If $X$ is Gorenstein projective and $Y$ is a direct summand of $X$ in $\K(R)$, then $Y$ is also Gorenstein projective.
   \item[(3-c)]
   Let $n$ be an integer.
   If $X$ is Gorenstein projective, then so is $X[n]$.
   \item[(3-d)]
   If $X$ is Gorenstein projective, then so are $\syz X$, $X^{\ast}$ and $\syz^{-}X$.
   \item[(3-e)]
   Let $X\xrightarrow{f}Y\xrightarrow{g}Z\xrightarrow{h}X[1]$ be a cohomologically ghost triangle in $\K(R)$.
   If $Z$ is Gorenstein projective, then $X$ is Gorenstein projective if and only if so is $Y$.
   \end{itemize}
   \end{enumerate}
\end{prop}

\begin{proof}
(1) Let $n\ge0$ be an integer.
Then
\begin{center}
$\pd X\le n$ $\Longleftrightarrow$ $\syz^n X\in\Add R$ $\Longleftrightarrow$ $C^{\bullet}(\syz^n X)\in\proj R$ $\Longleftrightarrow$ $\syz^n C^{\bullet}(X)\in\proj R$ $\Longleftrightarrow$ $\pd C^{\bullet}(X)\le n$.% $\Longleftrightarrow$ $\sup\{\pd C^i(X) \mid i\in\ZZ\}\le n$.
\end{center}
Here, the second equivalence follows from Theorem \ref{splitdef}, the third equivalence follows from Lemma \ref{Csyz}.
It follows from the above equivalences that the equality $\pd X=\sup\{\pd C^i(X)\mid i\in\ZZ\}$ holds.

(2) Since $B^{\bullet +1}(X)\approx\syz C^{\bullet +1}(X)$ by Remark \ref{boundry}(1), the module $B^{\bullet +1}(X)$ is Gorenstein projective by \cite[Lemma 3.10]{AB}.
As there exists an exact sequence $0\to \H^{\bullet}(X)\to C^{\bullet}(X)\to B^{\bullet +1}(X)\to0$, applying \cite[Lemma 3.10]{AB} again, we conclude that the module $\H^{\bullet}(X)$ is Gorenstein projective.

(3-a) This assertion is clear.

(3-b) Since $C^{\bullet}(Y)$ is a direct summand of $C^{\bullet}(X)$ in $\underline{\mod R}$ and $C^{\bullet}(X)$ is Gorenstein projective, $C^{\bullet}(Y)$ is also Gorenstein projective.

(3-c) Since $C^{\bullet}(X[n])=C^{\bullet +n}(X)$, the complex $X[n]$ is also Gorenstein projective.

(3-d) Since $C^{\bullet}(X)$ is Gorenstein projective, so are $\syz C^{\bullet}(X)$ and $\tr C^{\bullet}(X)$ by \cite[Lemma 3.10 and Proposition 3.8]{AB}.
By the isomorphisms $\syz C^{\bullet}(X)\approx C^{\bullet}(\syz X)$ and $\tr C^{\bullet}(X)\approx C^{-\bullet +1}(X^\ast)$, $\syz X$ and $X^\ast$ are Gorenstein projective.
Also, so is $\syz^{-}X$ as $\syz^{-}X=(\syz(X^\ast))^\ast$.

(3-e) Assume $X$ is Gorenstein projective.
Since $\Ext^i(\H^{\bullet}(X),R)=0=\Ext^i(\H^{\bullet}(Z),R)$ for all $i>0$ by the assertion (2), we have $\Ext^i(\H^{\bullet}(Y),R)=0$ for all $i>0$.
As the $R$-dual sequence $Z^\ast \xrightarrow{g^\ast} Y^\ast \xrightarrow{f^\ast} X^\ast \xrightarrow{h^\ast [1]}Z^\ast [1]$ is also a triangle in $\K(R)$, we obtain a commutative diagram with exact rows:
$$
\xymatrix@R-1pc@C-1pc{
&&\H^{-\bullet}(Z^\ast)\ar[rr]^-{\H^{-\bullet}(g^\ast)}\ar[dd]^-{\rho^{\bullet}_{Z}}&&\H^{-\bullet}(Y^\ast)\ar[rr]^-{\H^{-\bullet}(f^\ast)}\ar[dd]^-{\rho^{\bullet}_{Y}}&&\H^{-\bullet}(X^\ast)\ar[dd]^-{\rho^{\bullet}_{X}}\\
&&&&&&&&\\
0\ar[rr]&&\H^{\bullet}(Z)^\ast \ar[rr]^-{\H^{\bullet}(g)^\ast}&&\H^{\bullet}(Y)^\ast \ar[rr]^-{\H^{\bullet}(f)^\ast}&&\H^{\bullet}(X)^\ast\ar[rr]&&0.
}
$$
It follows from the injectivity of $\rho^{\bullet}_Z$ that $\H^{-\bullet}(g^\ast)$ is injective, and therefore $\H^{-\bullet}(f^\ast)$ is surjective.
Namely, the triangle $Z^\ast \xrightarrow{g^\ast} Y^\ast \xrightarrow{f^\ast} X^\ast \xrightarrow{h^\ast [1]} Z^\ast [1]$ is also cohomologically ghost triangle.
Applying the snake lemma to the above diagram, we have $\rho^{\bullet}_Y$ is bijective.
By the dual statement of Lemma \ref{ntrftriv}, the complex $Y^\ast$ is $\infty$-torsionfree.
Similarly, $Y$ is $\infty$-torsionfree since $Z^\ast \xrightarrow{g^\ast} Y^\ast \xrightarrow{f^\ast} X^\ast \xrightarrow{h^\ast [1]} Z^\ast [1]$ is also a cohomologically ghost triangle and the complexes $Z^\ast$ and $X^\ast$ are Gorenstein projective by the assertion (3-d).
Consequently, $Y$ is Gorenstein projective.

Conversely, assume that $Y$ is Gorenstein projective.
There exists a split complex $P\in\K(R)$ and a cohomologically ghost triangle $\syz Z\to X\oplus P\to Y\to \syz Z[1]$ from Lemma \ref{mawasu}.
It follows from the assertion (3-d) and the above argument that $X\oplus P$ is Gorenstein projective.
By the assertions (3-a) and (3-b), $X$ is also Gorenstein projective and we are done.
\end{proof}

For a finitely generated $R$-module $M$ and an integer $n\ge0$, the module $M$ has Gorenstein dimension at most $n$ if and only if the $n$th syzygy $\syz^n M$ is Gorenstein projective; see \cite[Theorem 3.13]{AB}.
A similar claim holds for objects of $\K(R)$.
It follows that the Gorenstein dimension of an object $X$ of $\K(R)$ can be computed from the Gorenstein dimensions of the $i$th cokernel $C^i(X)$ for $i\in\ZZ$.

\begin{thm}\label{Gdimca}
Let $X\in\K(R)$ and $n\ge0$ be an integer.
Then $\Gdim_R X\le n$ if and only if $\syz^n X$ is Gorenstein projective.
Also, the equality $\Gdim_R X=\sup\{\Gdim_R C^i(X) \mid i\in\ZZ\}$ holds.
\end{thm}

\begin{proof}
If $\syz^n X$ is Gorenstein projective, then $X$ has Gorenstein dimension at most $n$ by the definition.
Let us prove that the converse holds by induction on $n$.
There is nothing to prove when $n=0$.
Assume $n>0$.
Since $\Gdim X\le n$, there exists a cohomologically ghost triangle $X^\prime \to G\to X\to X^\prime [1]$ in $\K(R)$ such that $G$ is Gorenstein projective and $\Gdim X^\prime \le n-1$.
Applying Lemma \ref{mawasu} to this triangle repeatedly, we obtain a cohomologically ghost triangle $S^n_{X}\to S^{n-1}_{X^\prime}\to S^{n-1}_{G}\to S^n_{X}[1]$ with $S^n_X\approx \syz^n X$, $S^{n-1}_{X^\prime}\approx \syz^{n-1} X^{\prime}$ and $S^{n-1}_G\approx \syz^{n-1}G$ in $\K(R)$.
The induction hypothesis implies that $\syz^{n-1}X^\prime$ is Gorenstein projective.
By Proposition \ref{resol}(3), $S^{n-1}_{X^\prime}$ and $S^{n-1}_G$ are Gorenstein projective, so is $S^n_X$, and so is $\syz^n X$.

For any integer $n\ge0$, the following equivalences hold by the above argument and Lemma \ref{Csyz}:
\begin{eqnarray*}
\Gdim X\le n\text{ }&\Longleftrightarrow& \text{ }\syz^n X\text{ is Gorenstein projective }
\Longleftrightarrow\text{ }C^{\bullet}(\syz^n X)\text{ is Gorenstein projective }\\
&\Longleftrightarrow&\text{ }\syz^n C^{\bullet}(X)\text{ is Gorenstein projective }\Longleftrightarrow\text{ }\Gdim C^{\bullet}(X)\le n.
% $\Longleftrightarrow$ $\sup\{\pd C^i(X) \mid i\in\ZZ\}\le n$.
\end{eqnarray*}
These conclude that the equality $\Gdim_R X=\sup\{\Gdim_R C^i(X) \mid i\in\ZZ\}$ holds.
\end{proof}

%\begin{lem}\label{dimsup}
%Let $X\in\K(R)$.
%Then the following hold.  \\
%{\rm(1)} One has $\pd_RX=\sup\{\pd_R C^i(X) \mid i\in\ZZ \}$. \qquad
%{\rm(2)} One has $\Gdim_RX=\sup\{\Gdim_R C^i(X) \mid i\in\ZZ \}$.
%\end{lem}

A finitely generated $R$-module $M$ is said to have {\em complete resolution} if there exists an acyclic complex of finitely generated projective $R$-modules 
$$
\cdots \to P_1\to P_0\to P_{-1} \to P_{-2}\to \cdots
$$
whose $R$-dual is also exact such that $M\cong \Cok(P_1\to P_0)$.
%and the $R$-dual $\cdots \to P_{-1}^\ast \to P_0^\ast \to P_1^\ast \to \cdots$ is also exact.
It is well-known that a finitely generated $R$-module $M$ has complete resolution if and only if $M$ is Gorenstein projective; see \cite[Chapter 12]{LW}.
Sather-Wagstaff, Sharif and White \cite{SSW} showed that a finitely generated $R$-module $M$ satisfying the following condition is also Gorenstein projective; there exists an acyclic an acyclic complex of (finitely generated) Gorenstein projective $R$-modules 
$$
\cdots \to G_1\to G_0\to G_{-1} \to G_{-2}\to \cdots
$$
whose $R$-dual is also exact such that $M\cong \Cok(P_1\to P_0)$.
From the following theorem, the same argument holds for Gorenstein projective complexes.

\begin{thm}\label{SWet}
Let $X\in \K(R)$.
Then the following are equivalent.
\begin{enumerate}[\rm(1)]
   \item
   The complex $X$ is Gorenstein projective.
   \item
   For each integer $t$, there exists a cohomologically ghost triangle $X_{t+1}\to P_t\to X_t\to X_{t+1}[1]$ in $\K(R)$ with $P_t\in\Add R$ whose $R$-dual $X_{t}^\ast\to P_t^\ast\to X_{t+1}^\ast\to X_{t}^\ast[1]$ is also cohomologically ghost, such that $X\cong X_0$ in $\K(R)$.
   %There exist cohomologically ghost triangles $X_{t+1}\xrightarrow{g_{t+1}}P_t\xrightarrow{f_t}X_t\to X_{t+1}[1]$ in $\K(R)$ with $P_t\in\Add R$ for all integers $t$ such that $X\cong X_0$ in $\K(R)$ and the $R$-dual triangle $X_t^\ast \to P_t^\ast \to X_{t+1}^\ast \to X_t^\ast[1]$ is also cohomologically ghost for all integers $t$.
   \item
   For each integer $t$, there exists a cohomologically ghost triangle $X_{t+1}\to G_t\to X_t\to X_{t+1}[1]$ in $\K(R)$ with $G_t$ Gorenstein projective whose $R$-dual $X_{t}^\ast\to P_t^\ast\to X_{t+1}^\ast\to X_{t}^\ast[1]$ is also cohomologically ghost, such that $X\cong X_0$ in $\K(R)$.
   %There exist cohomologically ghost triangles $X_{t+1}\xrightarrow{g_{t+1}}G_t\xrightarrow{f_t}X_t\to X_{t+1}[1]$ in $\K(R)$ with $G_t$ Gorenstein projective for all integers $t$ such that $X\cong X_0$ in $\K(R)$ and the $R$-dual triangle $X_t^\ast \to G_t^\ast \to X_{t+1}^\ast \to X_t^\ast[1]$ is also cohomologically ghost for all integers $t$.
\end{enumerate}
\end{thm}

\begin{proof}
We first prove the implication $(1)\Rightarrow(2)$.
For any integer $t$ we take the triangle $\syz^{t+1}X\xrightarrow{g_{t+1}} F_t\xrightarrow{f_t}\syz^t X\to\syz^{t+1}X[1]$.
Then $f_t$ is cohomologically surjective.
Also, for the $R$-dual triangle $(\syz^t X)^\ast\xrightarrow{f_t^\ast}F_t^\ast\xrightarrow{g_{t+1}^\ast}(\syz^t X)^\ast[1]$ the morphism $f_t^\ast$ is cohomologically injective.
In fact, the diagram
$$
\xymatrix@R-1pc@C-1pc{
\H^{-\bullet}((\syz^t X)^\ast)\ar[rr]^-{\H^{-\bullet}(f_t^\ast)}\ar[dd]^-{\rho^\bullet_{\syz^t X}}&&\H^{-\bullet}(F_t^\ast)\ar[dd]^-{\rho^\bullet_{F_t}}\\
&&\\
\H^{\bullet}(\syz^t X)^\ast\ar[rr]^-{\H^{\bullet}(f_t)^\ast}&&\H^{\bullet}(F_t)^\ast
}
$$
is commutative and the natural maps $\rho^{\bullet}_{\syz^t X}$ and $\rho^{\bullet}_{F_t}$ are bijective since $\syz^t X$ is Gorenstein projective.
As $\H^{\bullet}(f_t)^\ast$ is injective, so is $\H^{-\bullet}(f_t^\ast)$.
Namely, $f_t^\ast$ is cohomologically injective and we obtain the desired triangles.

The implication $(2)\Rightarrow(3)$ clearly holds.
Suppose that the condition (3) is satisfied.
By the assumption in (3), we have the commutative diagram
$$
\xymatrix@R-1pc@C-1pc{
0\ar[rr]&&\H^{-\bullet}(X_t^\ast)\ar[rr]\ar[dd]^-{\rho^{\bullet}_{X_t}}&&\H^{-\bullet}(G_t^\ast)\ar[rr]\ar[dd]^-{\rho^{\bullet}_{G_t}}&&\H^{-\bullet}(X_{t+1}^\ast)\ar[rr]\ar[dd]^-{\rho^{\bullet}_{X_{t+1}}}&&0\\
&&&&&&&&\\
0\ar[rr]&&\H^{\bullet}(X_t)^\ast\ar[rr]&&\H^{\bullet}(G_t)^\ast\ar[rr]&&\H^{\bullet}(X_{t+1})^\ast
}
$$
with exact rows for all integers $t\in\ZZ$.
Since $\rho^{\bullet}_{G_t}$ is bijective for all $t\in \ZZ$, so is $\rho^{\bullet}_{X_t}$ for all $t$.
That is, $\Ext^1(C^{\bullet}(X_t), R)=\Ext^2(C^{\bullet}(X_t),R)=0$ for all $t\in\ZZ$.
From the long exact sequence
$$
0\to \Ext^1(\H^{\bullet}(X_t),R)\to\Ext^1(\H^{\bullet}(G_t), R)\to\Ext^1(\H^{\bullet}(X_{t+1}),R)\to\Ext^2(\H^{\bullet}(X_t),R)\to\cdots
$$
and the vanishing $\Ext^{>0}(\H^{\bullet}(G_t),R)=0$, we have $\Ext^k(\H^{\bullet}(X_t, R))\cong\Ext^{k-1}(\H^{\bullet}(X_{t+1}), R)\cong\cdots\cong\Ext^1(\H^{\bullet}(X_{t+k-1}),R)=0$ for all $k>0$ and $t\in\ZZ$.
Therefore, by Theorem \ref{fundex}, we have $\Ext^k(C^{\bullet}(X_t), R)\cong\Ext^{k-1}(C^{\bullet-1}(X_t),R)\cong\cdots\cong\Ext^{2}(C^{\bullet-k+2}(X_t),R)=0$ for all $k>2$ and $t\in\ZZ$.
We conclude that $\Ext^{>0}(C^{\bullet}(X_t),R)=0$ for all $t\in\ZZ$.
Next, we should show that $\Ext^{>0}(C^{\bullet}(X_t^\ast),R)=0$.
But this can be proved similarly from the dual diagram
$$
\xymatrix@R-1pc@C-1pc{
0\ar[rr]&&\H^{-\bullet}(X_{t+1})\ar[rr]\ar[dd]^-{\gamma^{\bullet}_{X_t}}&&\H^{-\bullet}(G_t)\ar[rr]\ar[dd]^-{\gamma^{\bullet}_{G_t}}&&\H^{-\bullet}(X_{t})\ar[rr]\ar[dd]^-{\gamma^{\bullet}_{X_{t+1}}}&&0\\
&&&&&&&&\\
0\ar[rr]&&\H^{\bullet}(X_{t+1}^\ast)^\ast\ar[rr]&&\H^{\bullet}(G_t^\ast)^\ast\ar[rr]&&\H^i(X_{t}^\ast)^\ast.
}
$$
\end{proof}

By \cite[Theorem 7.11]{Yos}, the pair $(\syz^{-n}, \syz^n)$ of functors from $\K(R)$ to itself is a adjoint pair for eact $n>0$. 
Namely, there is a functorial isomorphism $\Phi_{X,Y}:\Hom_{\underline{\K(R)}}(X,\syz^n Y)\xrightarrow{\sim}\Hom_{\underline{\K(R)}}(\syz^{-n}X,Y)$ for all $X, Y\in\underline{\K(R)}$.
We denote by $\psi_X^n:\syz^{-n}\syz^n X\to X$ the {\em counit morphism} of an object $X$ of $\K(R)$, i.e., $\underline{\psi_X^n}=\Phi_{\syz^n X, X}(\underline{1_{\syz^n X}})$.
In \cite[Section 10]{Yos}, the behavior of the counit morphism is studied.

The question of when syzygy modules become $n$-torsionfree is well-studied in representation theory; see \cite{AB, AR1, AR2}.
For a finitely generated $R$-module $M$, the fundamental result \cite[Proposition 2.21]{AB} due to Auslander and Bridger asserts that $\syz^n M$ is $n$-torsionfree if and only if there exists an exact sequence $0\to Y\to G\to M\to 0$ of finitely generated $R$-modules such that $Y$ has projective dimension at most $n$ and $\Ext^i_{R}(G,R)=0$ for all $1\le i\le n$.
The analogy of this theorem is the following theorem, which is one of the main results of this paper.
%シジジーがいつn-torsionfreeになるか、という問いはよく研究されている。Abは証明した、that,。この定理のアナロジーが次の定理であり、この論文の主結果の一つ。

\begin{thm}\label{ABapp}
Let $X\in \K(R)$ and $n$ be a positive integer.
Then $\syz^n X$ is $n$-torsionfree if and only if there exists a cohomologically ghost triangle $Y\to G\xrightarrow{\psi}X\to Y[1]$ in $\K(R)$ such that $G^\ast$ is $n$-torsionfree and $Y$ has projective dimension at most $n-1$.
\end{thm}

\begin{proof}
We first assume that the triangle $Y\to G\xrightarrow{\psi}X\to Y[1]$ mentioned in the theorem exists.
By Lemma \ref{mawasu}, there exists a cohomologically ghost triangle $\syz^{n}X\to\syz^{n-1}Y\to\syz^{n-1}G\to\syz^{n}X[1]$ up to $\Add R$ summands.
As this triangle splits, we have $\syz^n X\approx\syz^n G$.
The complex $\syz^n X$ is $n$-torsionfree.

Next, we consider the converse.
We prove by induction on $i$ that for all $0\le i\le n-1$ there are cohomologically ghost triangles $Y_i\to\syz^{-(i+1)}\syz^n X\xrightarrow{\psi^{i+1, n}_X}\syz^{n-(i+1)}X\to Y_{i}[1]$ and $Y_{i-1}\to P_{i}\to Y_i\to Y_{i-1}[1]$, where $P_i$ is a split complex and $Y_{-1}=0$.
Then the cohomologically ghost triangle $Y_{n-1}\to\syz^{-n}\syz^n X\to X\to Y_{n-1}[1]$ satisfies that $Y_{n-1}$ has projective dimension at most $n-1$ and $(\syz^{-n}\syz^n X)^\ast$ is $n$-torsionfree.
%If we can show this claim, then the cohomologically ghost triangle $B_{n-1}\to\syz^{-n}\syz^n X\to X\to B_{n-1}[1]$ is the the theorem is proved.
%Indeed, 

First we deal with the case $i=0$.
We take a right $\Add R$-approximation $p_0:Q_0\to\syz^{n-1}X$ of $\syz^{n-1}X$.
Since $\syz^n X$ is torsionless, we can take a left $\Add R$-approximation $q_0:\syz^{n}X\to P_0$ of $\syz^n X$ so that it is cohomologically injective.
Then, as the morphism $\binom{q_0}{0}:\syz^n X\to P_0\oplus Q_0$ is also a left $\Add R$-approximation of $\syz^n X$ and is cohomologically injective, the mapping cone of $\binom{q_0}{0}$ is isomorphic to the cosyzygy $\syz^{-1}\syz^n X$ in $\K(R)$ up to $\Add R$ summands.
We have the octahedron diagram
$$
\xymatrix@R-1pc@C-1pc{
\syz^n X\ar@{=}[dd]\ar[rr]^-{\binom{q_0}{0}}&&P_0\oplus Q_0\ar[dd]^-{(0, 1)}\ar[rr]&& \syz^{-1}\syz^n X \ar[rr]\ar[dd]&&\syz^n X[1]\ar@{=}[dd]\\
&&&\\
\syz^n X\ar[rr]&&Q_0\ar[dd]^-{0}\ar[rr]^-{p_0}&&\syz^{n-1}X\ar[dd]\ar[rr]&&\syz^n X[1]\\
&&&\\
&&P_0[1]\ar[dd]\ar@{=}[rr]&&P_0[1]\ar[dd]&&\\
&&&\\
&&(P_0\oplus Q_0)[1]\ar[rr]&&\syz^{-1}\syz^n X[1],&&
}
$$
%where UNNNUNNUNUN.
Put $Y_{-1}=0$ and $Y_0=P_0$.
We obtain the desired result.

We assume that $i$ is positive and there exists a cohomologically ghost triangle $Y_{i-1}\to\syz^{-i}\syz^n X\xrightarrow{\psi^{i, n}_X}\syz^{n-i}X\to Y_{i-1}[1]$.
Take a right $\Add R$-approximation $p_{i}:Q_i\to\syz^{n-(i+1)}X$.
Then there exists a cohomologically ghost triangle $\syz^{n-i}X\xrightarrow{r_i}Q_i\xrightarrow{p_i}\syz^{n-(i+1)}X\to\syz^{n-i}X[1]$.
By the assumption, the cosyzygy $\syz^{-i}\syz^{n}X$ is torsionless.% and Lemma ATODEATD, 
We can take a left $\Add R$-approximation $q_i:\syz^{-i}\syz^n X\to P_i$ of $\syz^{-i}\syz^nX$ so that it is cohomologically injective. 
Applying the nine lemma \cite[Exercise 10.2.6]{Wei}, we have the diagram whose rows and columns are triangles and whose squares commute except the lower right square for all $1\le i\le n-1$:
$$
\xymatrix@R-1pc@C-1pc{
Y_{i-1}\ar[rr]\ar[dd]&&P_i\ar[rr]\ar[dd]^-{\binom{1}{0}}&&Y_i\ar[dd]\ar[rr]&&Y_{i-1}[1]\ar[dd]\\
&&&&\\
\syz^{-i}\syz^n X\ar[rr]^-{\binom{q_i}{0}}\ar[dd]^-{\psi^{i, n}_X}&&P_i\oplus Q_i\ar[rr]\ar[dd]^-{(0, 1)}&&\syz^{-(i+1)}\syz^n X\ar[dd]^-{\psi^{i+1, n}_X}\ar[rr]&&\syz^{-i}\syz^n X\ar[dd]\\
&&&&\\
\syz^{n-i}X\ar[rr]^-{r_i}\ar[dd]&&Q_i\ar[rr]^-{p^\prime_i}\ar[dd]&&\syz^{n-(i+1)}X\ar[dd]\ar[rr]&&\syz^{n-i}X[1]\ar[dd]\\
&&&&\\
Y_{i-1}[1]\ar[rr]&&P_i[1]\ar[rr]&&Y_i[1]\ar[rr]&&Y_{i-1}[2].
}
$$
%t, in which allthe rows and columns are exact triangles and all the squares commute, exceptthe one marked "-" which anticommutes.

In the above diagram, since the morphisms $(0, 1)$ and $p^\prime_i$ are cohomologically surjective, so is $\psi^{i+1, n}_X$.
Similarly, the morphism $Y_{i-1}\to P_i$ is cohomologically injective.
Two cohomologically ghost triangles $Y_i\to\syz^{-(i+1)}\syz^n X\xrightarrow{\psi^{i+1, n}_X}\syz^{n-(i+1)}X\to Y_{i}[1]$ and $Y_{i-1}\to P_{i}\to Y_i\to Y_{i-1}[1]$ are obtained and we are done.
%Since $\chi^{i,n}_X$ and $\binom{q_i}{0}$ are cohomologically injective, so is the morphism $B_{i-1}\to P_i$.
%Similarly, the morphism $$
%$$
%\xymatrix@R-1pc@C-1pc{
%&&0\ar[rr]\ar[dd]&&F_0\ar[rr]\ar[dd]&&F_1\ar[rr]\ar[dd]&&\cdots\ar[rr]&&F_{n-1}\ar[rr]\ar[dd]&&B\ar[dd]\\
%&&&&&&&&&&&\\
%&&\syz^n X\ar@{=}[dd]\ar[rr]&&P_0\ar[rr]\ar[dd]&&P_1\ar[rr]\ar[dd]&&\cdots\ar[rr]&&P_{n-1}\ar[rr]\ar[dd]&&\syz^{-n}\syz^n X&&\\
%&&&&&&&&&&&\\
%&&\syz^n X\ar[rr]&&Q_0\ar[rr]&&Q_1\ar[rr]&&\cdots\ar[rr]&&Q_{n-1}\ar[rr]&&X
%}
%$$
\end{proof}

We obtain the following corollary, which is an analogy of a consequence of Auslander-Buchweitz theory \cite{ABu}.

\begin{cor}\label{AusBuchapp}
Let $X\in\K(R)$. Then $\Gdim_R X<\infty$ if and only if there exists a cohomologically ghost triangle $Y\to G\to X\to Y[1]$ in $\K(R)$ such that $\pd_R Y<\infty$ and $G$ is Gorenstein projective.
\end{cor}

\begin{proof}
Assume that the complex $X$ has finite Gorenstein dimension.
Put $n=\Gdim X$.
Thanks to Theorem \ref{Gdimca}, the $n$th syzygy $\syz^n X$ is Gorenstein projective, so it is $n$-torsionfree.
By Theorem \ref{ABapp}, there exists a cohomologically ghost triangle $Y\to G\to X\to Y[1]$ such that $\pd Y<n$ and $G^\ast$ is $n$-torsionfree.
Since $\Ext^{i}(C^{\bullet}(X), R)=0$ for all integers $1\le i\le n$, the fact that $\Gdim C^{\bullet}(X)\le n$ forces us to have $\Gdim C^{\bullet}(X)\le0$.
Thus $G$ is Gorenstein projective as desired.
The reverse implication follows from Theorem \ref{Gdimca}.
\end{proof}

We give some examples.
We see that our results recover known results for finitely generated $R$-modules.

\begin{ex}\label{ntrfex}
Let $M$ be a finitely generated $R$-module and take a finite projective presentation $P_1\to P_0\to M\to0$ of $M$.
%projective resolution $\cdots \to P_1 \xrightarrow{d^1_M} P_0 \to 0$ of $M$ with $P_i\in\proj R$.
We consider the complex $X_M=(\cdots\to 0\to P_1\xrightarrow{d^1_X}P_0\to 0\to \cdots)\in\K(R)$.
We note that the complex $X_M$ is $n$-torsionfree (resp. Gorenstein projective) if and only if the module $M$ is $n$-torsionfree (resp. Gorenstein projective).
Also, the equality $\Gdim X_M=\Gdim M$ holds by Theorem \ref{Gdimca}.
\begin{enumerate}[\rm(1)]
   \item 
   We recover \cite[Theorem 2.17 (a) $\Rightarrow$ (b)]{AB} from Theorem \ref{ntrfeq}.
   Let $n$ be an integer.
   Suppose that $M$ is $n$-torsionfree.
   Since $X_M$ is $n$-torsionfree, there exist cohomologically ghost triangles $X_i\to P_i\to X_{i+1}\to X_i[1]$ for $0\le i\le n-1$ such that $X_M\cong X_0$ in $\K(R)$ and the $R$-dual $X_{i+1}^\ast \to P_i^\ast \to X_i^\ast \to X_{i+1}^\ast[1]$ is also cohomologically ghost for all $0\le i\le n-1$ by Theorem \ref{ntrfeq}.
   Taking the $0$th cohomology, we obtain the exact sequence 
   $$
   0\to M=\H^0(X_M)\to \H^0(P_0)\to \H^0(P_1)\to\cdots\to \H^0(P_{n-1}),
   $$
   where $\H^0(P_i)$ is projective for all $0\le i\le n-1$.
   Since the $R$-dual $X_{i+1}^\ast \to P_i^\ast \to X_i^\ast \to X_{i+1}^\ast[1]$ is also cohomologically ghost and $P_i\in\Add R$ for all $0\le i\le n-1$, there exists an exact sequence 
   $$
   \H^0(P_{n-1})^\ast \to \cdots \to \H^0(P_1)^\ast \to \H^0(P_0)^\ast \to M^\ast \to 0.
   $$
   So the implication $(a)\Rightarrow(b)$ of \cite[Theorem 2.17]{AB} is obtained.
   \item 
   We prove that Gorenstein projective modules have a complete resolution.
   Suppose that $M$ is Gorenstein projective.
   Since $X_M$ is Gorenstein projective, by Theorem \ref{SWet}, there exist cohomologically ghost triangles $X_{t+1}\to P_t\to X_t\to X_{t+1}[1]$ with $P_t\in\Add R$ for all integers $t$ such that $X_M\cong X_0$ in $\K(R)$ and the $R$-dual triangle $X_t^\ast \to P_t^\ast \to X_{t+1}^\ast \to X_t^\ast[1]$ is also cohomologically ghost for all integers $t$.
   Then there exists an exact sequence
   $$
   \cdots\to \H^0(P_1)\to \H^0(P_0)\to \H^0(P_{-1})\to \H^0(P_{-2})\to\cdots
   $$
   such that $\Cok(\H^0(P_1)\to \H^0(P_0))\cong \H^0(X_M)\cong M$  and the $R$-dual
   $$
   \cdots\to \H^0(P_{-2})^\ast \to \H^0(P_{-1})^\ast \to \H^0(P_0)^\ast \to \H^0(P_1)^\ast \to\cdots
   $$
   is also exact.
   A complete resolution of $M$ is obtained.
   \item 
   Using Corollary \ref{AusBuchapp}, we prove that modules of finite Gorenstein dimension have the Auslander--Buchweitz approximation.% for modules of finite Gorenstein dimension.
   Suppose that $M$ has finite Gorenstein dimension.
   Since $X_M$ has finite Gorenstein dimension, there exists a cohomologically ghost triangle $Y\to G\to X_M\to Y[1]$ in $\K(R)$ such that $Y$ has finite projective dimension and $G$ is Gorenstein projective.
   Taking the $0$th cohomology, we obtain the exact sequence
   $$
   0\to \H^0(Y)\to \H^0(G)\to M\to 0.
   $$
   Then the module $\H^0(Y)$ has finite projective dimension and $\H^0(M)$ is a Gorenstein projective module by Proposition \ref{resol}.
   So the above exact sequence is an Auslander--Buchweitz approximation of the module $M$.
\end{enumerate}
%\begin{enumerate}[\rm(1)]
 %  \item 
 %  We consider the complex.
%\end{enumerate}
\end{ex}

Let $M$ be a finitely generated $R$-module.
Take two projective resolutions $\cdots\to P_1\to P_0\to0$ of $M$ with $P_i\in\proj R$ and $\cdots \to Q_2\to Q_1\to 0$ of $M^\ast$ with $Q_j\in\proj R^{\mathrm{op}}$.
Put $P_{-i}=Q_{i}^\ast$ for all $i>0$.
Then we have a complex
$$
F_M=(\cdots\to P_1\to P_0\to P_{-1} \to P_{-2} \to \cdots) \in \K(R).
$$
Kato \cite{KK2, KK} proved that the mapping $M\mapsto F_M$ gives a fully faithful functor $\underline{\mod R}\to \K(R)$.
We can calculate the Gorenstein dimension of $F_M$.
\begin{ex}
Under the above notation, we have
\begin{itemize}
   \item 
   $C^i(F_M)=\Cok(P_{i+1}\to P_{i})\cong \syz^i M$ for all $i\ge0$.
   \item 
   $C^{-i-1}(F_M)=\Cok(P_{-i}\to P_{-i-1})\cong \tr\syz^{i+1}\tr M =\syz^{-(i+1)}M$ for all $i\ge0$.
\end{itemize}
Note that $\Gdim \syz^i M\le\Gdim M$ for all integers $i\ge0$.
So the equality $\Gdim F_M=\sup\{\Gdim \syz^{-i}M \mid i\ge0\}$ holds by Theorem \ref{Gdimca}.
\end{ex}
%\begin{ex}\label{Gdimex}
%Let $M$ be a finitely generated $R$-module and take a finite projective presentation $P_1\to P_0\to M\to 0$ of $M$.
%By Theorem \ref{Gdimca}, we 
%\end{ex}

%%%%%%%%%%%%%%%%%%%%%%%%%%%%%%%%%%%%%%%%%%%%%%%%%%%%%%%%%%%%
\begin{ac}
%The author would like to thank his supervisor Ryo Takahashi for a lot of valuable discussions and advice.
The author is grateful to his supervisor Ryo Takahashi for his many valuable discussions and suggestions.
The author also thanks Osamu Iyama and Yuji Yoshino for their useful comments.
\end{ac}
%%%%%%%%%%%%%%%%%%%%%%%%%%%%%%%%%%%%%%%%%%%%%%%%%%%%%%%%%%%%%%

\end{document}